\documentclass[10pt]{article}
\usepackage[top=3.5cm, bottom=3cm, left=3cm, right=2.5cm]{geometry}

\usepackage{amssymb,amsmath}
\usepackage{amsthm}

\usepackage{amsfonts}
\usepackage{amscd}

\usepackage{subfigure}
\usepackage{graphicx}
\usepackage{epsfig}

\usepackage{esvect} %
\usepackage{enumerate}
\usepackage{enumitem}

\usepackage{dsfont}
\usepackage{mathtools}

\usepackage{hyperref}
\usepackage{nicefrac}
\theoremstyle{plain}
\newtheorem{thm}{Theorem}[section]
\theoremstyle{plain}
\newtheorem{lem}[thm]{Lemma}
\newtheorem{prop}[thm]{Proposition}
\newtheorem{cor}[thm]{Corollary}
\theoremstyle{definition}
\newtheorem{defi}[thm]{Definition}
\newtheorem{rem}[thm]{Remark}
\newtheorem{assumption}[thm]{Assumption}

\newtheorem{ex}[thm]{Example}

\newcommand{\ga}{\alpha}
\newcommand{\gb}{\beta}
\newcommand{\gd}{\delta}
\newcommand{\eps}{\ensuremath{\varepsilon}}

\renewcommand{\gg}{\gamma}
\newcommand{\gk}{\kappa}
\newcommand{\gl}{\lambda}
\newcommand{\go}{\omega}
\newcommand{\gs}{\sigma}

\newcommand{\gD}{\Delta}

\newcommand{\gO}{\Omega}

\newcommand{\cA}{\mathcal{A}}
\newcommand{\cB}{\mathcal{B}}
\newcommand{\cC}{\mathcal{C}}
\newcommand{\cD}{\mathcal{D}}

\newcommand{\cF}{\mathcal{F}}
\newcommand{\cG}{\mathcal{G}}
\newcommand{\cH}{\mathcal{H}}
\newcommand{\cI}{\mathcal{I}}

\newcommand{\cL}{\mathcal{L}}
\newcommand{\cM}{\mathcal{M}}
\newcommand{\cN}{\mathcal{N}}
\newcommand{\cO}{\mathcal{O}}
\newcommand{\cP}{\mathcal{P}}
\newcommand{\cQ}{\mathcal{Q}}

\newcommand{\cU}{\mathcal{U}}

\newcommand{\cW}{\mathcal{W}}
\newcommand{\cX}{\mathcal{X}}
\newcommand{\cY}{\mathcal{Y}}
\newcommand{\cZ}{\mathcal{Z}}

\newcommand{\bE}{\mathbb{E}}

\newcommand{\bN}{\mathbb{N}}

\newcommand{\bP}{\mathbb{P}}

\newcommand{\bR}{\mathbb{R}}

\newcommand{\bT}{\mathbb{T}}

\newcommand{\mfd}{\mathfrak{d}}

\newcommand{\mfN}{\mathfrak{N}}

\newcommand{\mfW}{\mathfrak{W}}

\newcommand{\LHS}{\cL_{\textnormal{HS}}}

\newcommand{\be}{\begin {equation}}
\newcommand{\ee}{\end  {equation}}
\newcommand{\bee}{\begin {equation*}}
\newcommand{\eee}{\end {equation*}}
\newcommand{\ol}{\overline}

\newcommand{\floor}[1]{\left\lfloor #1 \right\rfloor}
\newcommand{\laplace}{\bigtriangleup}
\newcommand{\abs}[1]{\left\vert#1\right\vert}
\newcommand{\var}{{\rm Var}}

\newcommand{\tr}{\text{Tr}}

\newcommand{\KL}{{Karhunen-Lo\`{e}ve }}
\newcommand{\MC}{Monte Carlo}
\newcommand{\MLMC}{multilevel \MC}
\newcommand{\MMLMC}{Multilevel \MC}

\newcommand{\bracket}[2]{\left\langle #1, #2 \right\rangle}
\newcommand{\dbracket}[2]{\left\langle\langle #1, #2 \right\rangle\rangle}

\newcommand{\norm}[1]{\left\|#1\right\|}
\newcommand{\set}[1]{\left\{#1\right\}}
\newcommand{\E}[1]{\bE\left(#1\right)}
\newcommand{\LR}[1]{\left(#1\right)}

\let\inf\relax \DeclareMathOperator*\inf{\vphantom{p}inf}

\usepackage[dvipsnames]{xcolor}

\usepackage{todonotes}
\newcommand\hcancel[2][black]{%
  \relax\ifmmode%
    \setbox0=\hbox{$#2$}%
    \else\setbox0=\hbox{#2}\fi%
  \rlap{\raisebox{.45\ht0}{\textcolor{#1}{\rule{\wd0}{1pt}}}}#2}
\usepackage{pgfplots}
\pgfplotsset{compat=1.18}

\def\tI{\textrm{I}}
\def\tII{\textrm{II}}
\def\tIII{\textrm{III}}
\def\tIV{\textrm{IV}}
\def\tV{\textrm{V}}
\def\tVI{\textrm{VI}}

\def\add#1{#1}
\def\del#1{}
\def\rdel#1{}

\def\correctSpace{{\del{\cL_1}}{\add{\LHS}}}

\begin{document}

\title{An Antithetic Multilevel Monte Carlo-Milstein Scheme\\ for Stochastic
  Partial Differential Equations \add{with non-commutative noise}}

\author{
  Abdul-Lateef Haji-Ali \footnote{Maxwell
    Institute, School of Mathematical and Computer Sciences, Heriot-Watt
    University, Edinburgh. Email:
	\href{mailto:a.hajiali@hw.ac.uk}{a.hajiali@hw.ac.uk}}
	\and
	Andreas Stein
	\footnote{Seminar for Applied Mathematics, ETH Zürich. Email: \href{mailto:andreas.stein@sam.math.ethz.ch}{andreas.stein@sam.math.ethz.ch}
	}
}

\date{\today}

\maketitle

\begin{abstract}
  We present a novel multilevel Monte Carlo approach for estimating quantities of interest for stochastic partial differential equations (SPDEs) \add{with non-commutative noise}. Drawing inspiration from \cite{giles2014antithetic}, we extend the antithetic Milstein scheme for finite-dimensional stochastic differential equations to Hilbert space-valued SPDEs. Our method has the advantages of both Euler and Milstein discretizations, as it is easy to implement and does not involve intractable L\'evy area terms. Moreover, the antithetic correction in our method leads to the same variance decay in a MLMC algorithm as the standard Milstein method, resulting in significantly lower computational complexity than a corresponding MLMC Euler scheme. Our approach is applicable to a broader range of non-linear diffusion coefficients and does not require any commutative properties. The key component of our MLMC algorithm is a truncated Milstein-type time stepping scheme for SPDEs, which accelerates the rate of variance decay in the MLMC method when combined with an antithetic coupling on the fine scales. We combine the truncated Milstein scheme with appropriate spatial discretizations and noise approximations on all scales to obtain a fully discrete scheme and show that the antithetic coupling does not introduce an additional bias.
\end{abstract}

{\bf Keywords:} Stochastic Partial Differential Equations, Multilevel Monte Carlo, Milstein Scheme, Variance Reduction, Antithetic Variates.

{\bf Subject classifiction:} 65C05, 65C30, 65M12.

\section{Introduction}
\label{sec:intro}

Stochastic partial differential equations (SPDEs) are encountered in a range
of applications spanning natural sciences, engineering, and finance. Examples
include stochastic epidemic compartment models \cite{nguyen2020analysis} and
the valuation of forward contracts in interest rate or energy markets
\cite{CT07, BK08, BB14}. However, a common challenge in these applications is
that SPDEs do not possess a closed-form solution and must therefore be
approximated numerically. Fortunately, numerous numerical schemes for
approximating various types of SPDEs have been established. A non-exhaustive
list of references provide strong approximation results for numerous SPDEs
with different approximations schemes in space and Euler \cite{JP09, B10,
  JKN11, DHP12, kruse2014strong, KLL15, GSS16, ACLW16, BS2019stochastic} or
Milstein
\cite{BL12b,jentzen:milstein-SPDE,barth:milstein-SPDE,hallern:milstein-SPDE}
discretizations in time, while others \cite{D09, KLS15, BSt18} offer a weak
error analysis.

\del{Once the `pathwise' approximations are obtained, they are utilized in sampling-based approaches to estimate specific quantities of interest within the SPDE model.
Monte Carlo (MC) methods are a common choice for this purpose.} %
\add{Sampling-based approaches, such as Monte Carlo (MC), which estimate expectations of
  quantities of interest depending on the SPDE model, utilize sufficiently accurate approximations of pathwise samples for that purpose.} %
\del{However, due to the low regularity of the model, standard MC approaches may become prohibitively expensive even for comparatively simple SPDEs.}
\add{However, due to the low regularity of the model, such accurate samples
  are expensive to compute, which compounds the sampling cost and renders
  Monte Carlo prohibitively inefficient.} %
In addition, higher-order \del{schemes}\add{sampling methods} for
\del{discretizing}\add{resolving} the stochastic space, such as stochastic
Galerkin or Quasi-Monte Carlo methods, are not \del{feasible}\add{suitable}
due to the limited regularity of the model.
Thus, the multilevel Monte Carlo (MLMC) method \cite{giles2015multilevel}
\del{seems to be the only viable}\add{has emerged as a good} option to
accelerate the estimation of expectations for SPDEs. This approach has been
studied in the context of SPDEs in \cite{BLS13, barth2012multilevel,
  giles2012stochastic, lang2016note} \add{with Euler and Milstien
  discretizations in time}.

One common drawback of the \del{schemes}\add{MLMC estimators} presented in
\cite{BLS13, barth2012multilevel, lang2016note} is that they rely on a simple
Euler discretization in time, which leads to slow temporal convergence rates.
In contrast, the authors of \cite{giles2012stochastic} propose a MLMC-Milstein
\del{scheme}\add{estimator} that uses a finite difference approximation in
space to accelerate temporal convergence. However, their SPDE model is
considerably simplified, as it is only driven by a one-dimensional Brownian
motion. Consequently, it is not necessary to simulate L\'evy area terms.
The simulation of these terms is a substantial issue \add{when using Milstein
  schemes} even for three-dimensional stochastic differential equations (SDEs)
without certain commutativity conditions on the diffusion term. Moreover, the
problem is exacerbated for infinite-dimensional driving noise, which is the
natural setting for SPDEs.

\subsection{Contributions}
\label{sec:contributions}

The objective of this research article is to address the previously mentioned
issues by introducing an \textit{antithetic \MLMC-Milstein scheme} for
parabolic SPDEs \add{with non-commutative noise}. Our work is based on the
antithetic MLMC scheme for SDEs presented in \cite{giles2014antithetic} and
offers several advantages.
Firstly, under natural assumptions, our scheme achieves higher-order convergence rates, similar to those of the 'standard' Milstein scheme. Secondly, the antithetic approach eliminates the need to sample L\'evy area terms, making the scheme easy to implement.
Our complexity analysis demonstrates that the proposed MLMC algorithm can significantly reduce computational time by several orders of magnitude.
Finally, we extend the results for SDEs from \cite{giles2014antithetic} by allowing for unbounded, random initial conditions and not requiring a global Lipschitz condition on the Milstein correction term.

\subsection{Outline}

The article is structured as follows: first, in Section~\ref{sec:prelim}, we provide the necessary notation and background on functional analysis, infinite-dimensional Wiener processes, and parabolic SPDEs. In Section~\ref{sec:approximation}, we propose discretization methods for the spatial, stochastic, and temporal domains of the SPDE. The main contribution of our paper is presented in Section~\ref{sec:anti-mlmc}, where we introduce the antithetic Milstein scheme and prove its expected variance decay in Theorem~\ref{thm:MainRes}. We then analyze the complexity of the associated antithetic MLMC Milstein scheme in Section~\ref{sec:MLMC} and present numerical experiments in Section~\ref{sec:numerics} to complement our theoretical analysis. All proofs are provided in an appendix for clarity.

\section{Preliminaries}
\label{sec:prelim}

\subsection{Basic Notation}
Let $(\cY, \left\|\cdot\right\|_\cY)$ and $(\cZ, \left\|\cdot\right\|_\cZ)$ be two Banach spaces.
The Borel $\gs$-algebra of $\cY$ is generated by the open sets in $\cY$ and denoted by $\cB(\cY)$.
We further denote by $\cL(\cY, \cZ)$ and $\cL(\cY)$ the set of linear bounded operators $O: \cY\to  \cZ$ and $O: \cY\to
\cY$, respectively.
For any (bounded or unbounded) operator $O: \cY\to \cZ$, we denote its adjoint by $O^*:\cZ\to \cY$.
Let $\cY_0\subseteq\cY$ be an open subset and let $F:\cY\to \cZ$ be a twice Fr\'echet differentiable mapping on $\cY_0$. The first two Fr\'echet derivatives of $F$ are given by
$F': \cY_0\to \cL(\cY, \cZ)$  and
$F'': \cY_0\to \cL(\cY, \cL(\cY, \cZ))\simeq \cL(\cY\times\cY, \cZ)$.
For the remainder of this article, $C>0$ denotes a generic positive constant which may change from one line to another. The dependency of $C$ on certain parameters is made explicit if necessary. \add{Moreover, let \(\bN\) denote the set of natural numbers excluding zero.}

\subsection{Hilbert-Schmidt Operators and RKHS}
Throughout this article, we consider two separable Hilbert spaces $(U,(\cdot,\cdot)_U)$ and $(H,(\cdot,\cdot)_H)$.
The space of \textit{Hilbert-Schmidt operators} \add{\cite[Appendix A]{PZ07}} on $U$
is given by
$$\LHS(U,H):=\{O\in \cL(U,H)|\; \|O\|^2_{\LHS(U,H)}:=\sum_{k\in\bN} \|O u_k\|^2_{H} < \infty \},$$
where $(u_k,k\in\bN)$ is some orthonormal basis of $U$. Recall that $(\LHS(U,H), \norm{\cdot}_{\LHS(U,H)})$ is separable, while this is in general not true for $\cL(U, H)$.
Further, $\LHS(U,H)$ is a Hilbert space equipped with the tensor product
\begin{equation*}
	(O_1, O_2)_{\LHS(U,H)}:=\sum_{k\in\bN} (O_1 u_k, O_2 u_k)_H, \quad O_1,O_2\in\LHS(U,H).
\end{equation*}
The tensor product of $U$ and $H$ is denoted by $(U\otimes H, (\cdot,\cdot)_{U\otimes H})$. For $\phi\in U$ and $\psi\in H$ we associate to $\phi\otimes\psi\in U\otimes H$ the rank one operator $O_{\phi,\psi}\in \LHS(U, H)$, such that $O_{\phi,\psi}u = (\phi, u)_U\psi$ for all $u\in U$.
Thus, we use the identification $U\otimes H\simeq \LHS(U,H)$, as $U\otimes H$ and $\LHS(U,H)$ are isometrically isomorphic.

We denote by $\cL_1(U)$ the space of all trace class operators on $U$, and by $\cL_1^+(U)$ the subset of all non-negative, self-adjoint operators on $U$ with finite trace. The trace of $Q\in \cL_1^+(U)$ is denoted by $\tr(Q)<\infty$.
For any $Q\in \cL_1^+(U)$, the Hilbert-Schmidt theorem yields that the ordered eigenvalues $\eta_1\ge\eta_2\ge\dots\ge0$ are non-negative with zero as only accumulation point, and the corresponding eigenfunctions $(e_k,k\in\bN)\subset U$ form an orthonormal basis of $U$.
The \textit{square-root} of $Q\in \cL_1^+(U)$ is defined via
\bee
Q^{\nicefrac{1}{2}}\phi:=\sum_{k\in\bN}\sqrt{\eta_k}(\phi,e_k)_Ue_k,\quad\phi\in U.
\eee
Since $Q^{\nicefrac{1}{2}}$ is not necessarily injective, the \textit{pseudo-inverse} of $Q^{\nicefrac{1}{2}}$ is given by
\bee
Q^{-1/2}\varphi:=\phi,\quad\text{if $Q^{\nicefrac{1}{2}}\phi=\varphi$\;
	and $\|\phi\|_U=\inf\set{
		\|\varphi\|_U:\;\text{$ \varphi\in U$ is such that $Q^{\nicefrac{1}{2}}\varphi=\phi$}}$}.
\eee
We define \textit{reproducing kernel Hilbert space} (RKHS) associated to $Q$ as the set $\cU:=Q^{\nicefrac{1}{2}}(U)$ equipped with the scalar-product
\bee
(\varphi_1,\varphi_2)_\cU:=(Q^{-1/2}\varphi_1,Q^{-1/2}\varphi_2)_U,\quad \varphi_1,\varphi_2\in\cU.
\eee
Note that $(\sqrt{\eta_k} e_k,k\in\bN)$ forms an orthonormal system in $\cU$, hence
\bee
\|O\|^2_{\LHS(\cU,H)}=\sum_{k\in\bN}\eta_k \|Oe_k\|_H^2,\quad O\in \LHS(\cU,H).
\eee

\subsection{Martingales on Hilbert Spaces}

We consider a filtered probability space $(\gO,\cF,\bP,(\cF_t,t\ge0))$ with normal filtration and a finite time interval $\bT=[0,T]$.
The Lebesgue-Bochner space of all $p$-integrable, $H$-valued random variables is given as
\bee
L^p(\gO; H):=\set{Y:\gO\to H \text{ is  measurable with $\|Y\|_{L^p(\gO; H)}:=\E{\norm{Y}_H^p}^{1/p}<\infty$}},\quad p\in[1,\infty).
\eee
Solutions to stochastic partial differential equations (SPDEs) are defined as predictable $H$-valued processes.
The \textit{predictable $\gs$-algebra} $\cP_\bT$ is the smallest $\gs$-field on $\gO\times\bT$ containing all sets of the form $\cA\times (s,t]$, where $\cA\in\cF_s$ and $s,t\in\bT$ with $s<t$.
An $H$-valued stochastic process $Y:\gO\times\bT\to H$ is called \textit{predictable} if it is a $\cP_\bT/\cB(H)$-measurable mapping.
The set of all square-integrable, $H$-valued predictable processes is denoted by
\be\label{eq:predict_processes}
\cX_{\bT}:=\set{X:\gO\times\bT\to H\,\big|\,\text{$X$ is predictable and $\sup_{t\in\bT} \bE(\|X(t)\|_H^2)<\infty$} }.
\ee
All appearing equalities and estimates involving stochastic terms are in the path-wise sense and are assumed to hold almost surely, thus we omit the stochastic argument $\go\in\gO$ for notational convenience.

\begin{defi}\add{\cite[Chapter 8]{PZ07}}
	Let $(e_k,k\in\bN)$ be an arbitrary orthonormal basis of $U$ and denote $\cM^2(U)$ the set of all square-integrable, $U$-valued martingales.
	\begin{enumerate}
		\item For $Y\in\cM^2(U)$, denote by $\bracket{Y}{Y}:\gO\times\bT\to\bR$ the unique predictable (quadratic variation) process, such that $\bT\ni t\mapsto \norm{Y(t)}_U^2- \bracket{Y}{Y}_t$ is a real-valued martingale. The covariation of two martingales $Y,Z\in\cM^2(U)$ is given by the polarization identity
		\begin{equation*}
			\bracket{Y}{Z}:=\frac{1}{2}\LR{\bracket{Y+Z}{Y+Z}-\bracket{Y}{Y}-\bracket{Z}{Z}}.
		\end{equation*}

		\item  The \emph{operator-valued angle bracket process}  $\dbracket{Y}{Y}:\gO\times \bT\to \cL_1^+(U)$ of $Y\in\cM^2(U)$ is defined as
		\begin{equation*}
			\dbracket{Y}{Y}:\gO\times \bT\to\LHS(U), \quad
			t\mapsto \sum_{k,l \in\bN} \bracket{(Y(\cdot), e_k)_U}{(Y(\cdot), e_l)_U}_t e_k\otimes e_l.
		\end{equation*}
	\end{enumerate}
\end{defi}

It holds that $\dbracket{Y}{Y}$ is the unique process such that $\bT\ni t\mapsto Y(t)\otimes Y(t)-\dbracket{Y}{Y}_t$ is an $\cL_1(U)$-valued martingale. Further, there exists a unique process $\cQ_Y:\gO\times \bT\to \cL_1^+(U)$, called the \emph{martingale covariance} of $Y$,  such that
\begin{equation}\label{eq:bracket_integral}
	\dbracket{Y}{Y}_t=\int_0^t \cQ_Y(s)\, d\bracket{Y}{Y}_s, \quad t\in\bT,
\end{equation}
see e.g. \cite[Theorem 8.2/Definition 8.3]{PZ07}.
We consider $H$-valued stochastic integrals $\int_0^tG(s)dY(s)$ with predictable,
operator-valued integrands $G:\gO\times\bT \del{\mapsto \cL}\add{\to \LHS} (\cU, H)$ such that
$G\circ \cQ_Y^{\nicefrac{1}{2}}:\gO\times\bT \del{\mapsto}\add{\to} \LHS(U, H)$\add{, see
\cite[Section 8.2 and 8.3]{PZ07} for the formal construction of such
stochastic integrals.}

\subsection{Wiener Process on a Hilbert Space}
\label{sec:Wiener}

\begin{defi}\cite[Definition 2.1.9]{prevot2007concise}
	Let $Q\in\cL_1^+(U)$.
	A $U$-valued stochastic process $W=(W(t),t\in\bT)$ on $(\gO, \cF, \bP)$ is called a \textit{$Q$-Wiener process} if
	\begin{itemize}
		\item $W(0)=0$,
		\item $W$ has $\bP$-almost surely continuous trajectories,
		\item $W$ has independent increments, and
		\item for all $0\le s \le t\le T$ there holds that $W(t)-W(s)\sim \cN(0, (t-s)Q)$.
	\end{itemize}
\end{defi}

For any $Q$-Wiener process there holds the identity
\bee
\bE( (W(t)-\bE(W(t)),\phi)_U (W(t)-\bE(W(t)),\psi)_U )=t(Q\phi,\psi)_U,\quad \phi,\psi\in U,\;t\in\bT.
\eee
It follows that $\bracket{W}{W}_t= t\,\tr(Q)$ and $\dbracket{W}{W}_t = t Q$ (note that $\cQ_Y=Q\, \tr(Q)^{-1}$ in~\eqref{eq:bracket_integral} is constant with respect to $t$ in this case).
Further, recall that $W$ admits the \textit{\KL expansion}
\be\label{eq:KL}
W(t)=\sum_{k\in\bN} (W(t),e_k)_Ue_k
\stackrel{d}{=}\sum_{k\in\bN} \sqrt{\eta_k}w_k(t)e_k,\quad t\in\bT,
\ee
where the relation $\stackrel{d}{=}$ signifies equality in distribution and $(w_k,k\in\bN)$ is a sequence of real-valued and independent standard Brownian motions.

\subsection{Stochastic Partial Differential Equations}
\label{sec:SPDEs}

We consider the stochastic partial differential equation (SPDE)
\be\label{eq:spde}
dX(t)=(AX(t)+F(X(t)))dt+G(X(t))dW(t), \quad X(0)=X_0,
\ee
where $A: D(A)\subset H\to H$ is a densely defined and unbounded linear (differential) operator.
The initial value $X_0$ is a $H$-valued random variable, $W$ is a $Q$-Wiener process, and
the coefficients $F$ and $G$ in Eq.~\eqref{eq:spde} are (possibly) non-linear measurable mappings  $F:H\to H$ and $G:H\to \LHS(\cU,H)$, respectively.
Throughout this article we will assume that $(-A)$ is self-adjoint, positive definite and boundedly invertible.
Consequently, the eigenvalues $(\lambda_n, n\in\bN)$ of $(-A)$ are positive, non-decreasing and only accumulate at infinity, with the corresponding eigenfunctions $(f_n, n\in\bN)$ spanning an orthonormal basis of $H$.

By the Hille-Yosida Theorem, $A$ is the generator of an analytic semigroup $S=(S(t), t\ge 0)\subset \cL(H)$ (see e.g. \cite[Appendix B.2]{kruse2014strong}).
The fractional powers of $(-A)$, given by
\begin{equation*}
	(-A)^{\ga}\, v:= \sum_{n\in\bN} \gl_n^\ga (v,f_n)_{H}f_n\quad v\in H,
\end{equation*}
are well-defined for any $\ga\in\bR$. Moreover, $(-A)^{\ga}:D((-A)^{\ga})\to H$ is
a closed operator, with $D((-A)^{\ga})$ being dense in $H$ \add{for all \(\ga \geq 0\)}
(see e.g. \cite[Chapter 2, Theorem 6.8]{pazy1983semigroups}).
We define the Hilbert space $\dot H^\ga:=D((-A)^{\nicefrac \ga 2})$ equipped with the inner product $(\cdot, \cdot)_\ga:=((-A)^{\frac \ga 2}\cdot, (-A)^{\frac \ga 2}\cdot)_H$, which will in turn be used to quantify smoothness of solutions to~\eqref{eq:spde}.

\begin{ex}\label{ex:dotH}
	Let $H=L^2(\cD)$ for on a bounded, convex domain $\cD\subset\bR^d$, and let $A = \laplace $ be the Laplace operator with zero Dirichlet boundary conditions on $\cD$. It then holds that $\dot H^2=D((-A))=H^2(\cD)\cap H_0^1(\cD)$.
	More generally, it holds for $\ga\in[1,2]$ that
	$\dot H^\ga=D((-A)^{\nicefrac{\ga}{2}})=H^\ga(\cD)\cap H_0^1(\cD)$, see \cite[Proposition 4.1]{bonito2015numerical}.
\end{ex}

We formulate suitable, but natural assumptions on the initial value and the coefficients of the SPDE~\eqref{eq:spde} in the following. We also repeat the above conditions on $A$ for the reader's convenience.

\begin{assumption}
	\add{Fix \(\ga \geq 1\) and assume that:}
	\begin{enumerate}[label=(\roman*)] \label{ass:SPDE}
        \item The operator $A: D(A)\subset H\to H$ is self-adjoint, densely defined in $H$ and the infinitesimal generator of an analytic semigroup \(S=(S(t), t\ge 0)\subset \cL(H)\)\add{, in other words, \(S:\:\bT\to \cL(H),\, t \mapsto e^{t A}\)}.
		Moreover, $(-A):D(A)\to H$ is boundedly invertible, i.e. $0\in\rho(A)$, where $\rho(A)$ is the resolvent set of $A$.

		\item\label{item:IC} $X_0\in L^8(\gO; \del{H}\add{\dot H^{\add{\ga}}})$ is a $\cF_0$-measurable random variable.

		\item\label{item:FG-diff} The mappings $F:H\to H$ and $G:H\to \LHS(\cU,H)$ are twice Fr\'echet differentiable on $H$ with bounded derivatives, i.e. there is a $C>0$ such that for all $v\in H$ there holds
		\begin{align*}
			\norm{F'(v)}_{\cL(H)} + \norm{F''(v)}_{\cL(H\times H, H)} &\le C \\
			\norm{G'(v)}_{\cL(H, \LHS(\cU,H))} + \norm{G''(v)}_{\cL(H\times H, \LHS(\cU,H))} &\le C.
		\end{align*}

		\item\label{item:smooth-coefficients} There are constants $C>0$ such that for all $v\in \dot H^\ga$ there hold the linear growth bounds
		\begin{align*}
			\|F(v)\|_{\dot H^\ga}
			+\|G(v)\|_{\LHS(\cU, \dot H^\ga)}
			&\le C(1+\|v\|_{\dot H^\ga}), \\
			\|G'(v) \|_{\cL(\dot H^\ga, \LHS(\cU, \dot H^\ga))}
			&\le C.
		\end{align*}
	\end{enumerate}
\end{assumption}

\begin{rem}
  We require $X_0\in L^8(\gO; \del{H}\add{\dot H^{\add{\ga}}})$, rather than $X_0\in L^2(\gO;
  \del{H}\add{\dot H^{\add{\ga}}})$, in Item~\ref{item:IC} for some technical steps in the
  proofs (cf. Lemma~\ref{lem:semi-error} in the Appendix), as we apply
  Hölder's inequality to obtain suitable mean-square error bounds.
\end{rem}

Mild solutions to SPDEs are characterized by path-wise identities that hold almost surely as follows:

\begin{defi}{\cite[Chapter 9]{PZ07}} \label{def:solutions}
	Let $\cX_\bT$ be as in~\eqref{eq:predict_processes}.
	A process $X\in\cX_\bT$ is called a \textit{mild solution} to Eq.~\eqref{eq:spde} if for all $t\in\bT$ there holds $\bP$-a.s.
	\be \label{eq:mild}
	X(t)=S(t)X_0+\int_0^tS(t-s)F(X(s))ds+\int_0^tS(t-s)G(X(s))dW(s).
	\ee
\end{defi}

\del{In Eq.~\eqref{eq:mild}, $S:\bT\to \cL(H)$ is the semigroup generated by $(-A)$, thus $S(t)=e^{-tA}$ and Eq.~\eqref{eq:mild} may be interpreted as a \textit{variation-of-constants} formula.}
\add{Recalling that \(S(t) = e^{-tA}\), Eq.~\eqref{eq:mild} may then be interpreted as a \textit{variation-of-constants} formula.}
Well-posedness of \eqref{eq:spde} in the mild sense, and regularity of solutions has been investigated under suitable assumptions on $F, G$ and $X_0$, see e.g. \cite[Theorems 9.15 and 9.29]{PZ07} or \cite[Chapters 2.4-2.6]{kruse2014strong}. We condense the main results in the following statement.

\begin{thm}\label{thm:well-posedness}
  Under Assumption~\ref{ass:SPDE}, there exists a unique mild solution
  $X\in\cX_\bT$ to \eqref{eq:spde}, such that for all $p\in (0, 8]$ and $\gk\in[0,
  \ga)$ it holds that
	\begin{equation*}
		\sup_{t\in\bT}\bE(\|X(t)\|_{\dot H^\ga}^p)<\infty \quad\text{and}\quad
		\sup_{t, s\in\bT} \frac{\bE(\|X(t)-X(s)\|_{\dot H^{\gk}}^p)^{\nicefrac{1}{p}}}
		{|t-s|^{\min(\nicefrac{1}{2}, \nicefrac{(\ga-\gk)}{2})}}<\infty.
	\end{equation*}
\end{thm}

\section{Pathwise Approximations}
\label{sec:approximation}

\subsection{Spatial Discretization}
To derive a spatial approximation based\add{, we follow \cite[Section
  3.2]{kruse2014strong} and} define $V:=\dot H^1=D((-A)^{\nicefrac{1}{2}})$ and
  consider the bilinear form
\begin{equation}
  B:V\times V\to \bR,\quad \del{(v,w)\mapsto ((-A)^{\nicefrac{1}{2}}v, (-A)^{\nicefrac{1}{2}}w)}
  \add{B(v,w) :=  (v,w)_{1} = ((-A)^{\nicefrac{1}{2}}v, (-A)^{\nicefrac{1}{2}}w)_{H}}.
\end{equation}
In Example~\ref{ex:dotH}, where $A$ is the Laplacian with zero Dirichlet
boundary conditions on a convex domain $\cD\subset\bR^d$, we have $V=H_0^1(\cD)$,
and $B(v,w)=(\nabla v, \nabla w)_H$.

We replace $V$ by a finite dimensional subspace $V_N$ with $N:=\dim(V_N)\in\bN$.
This encompasses several spatial approximations, for instance spectral Galerkin methods, where $N$ is number of terms in expansion, and finite element methods, where the mesh refinement parameter $h>0$ is related to $N$ via $N=\cO(h^{-d})$.
We introduce the \emph{discrete} operator $A_N:V_N\to V_N$ by
\begin{equation}\label{eq:operator-discrete}
	(-A_N v_N, w_N)_{\add{1}} = B(v_N, w_N),\quad v_N, w_N\in V_N.
\end{equation}
Then, $(-A_N)$ generates an analytic semigroup $(S_N(t), t\ge0)$ of linear operators $S_N(t):V_N\to V_N$ via $S_N(t):=\exp(-tA_N)$.
Let $P_N:H\to V_N$ be the $H$-orthogonal projection onto $V_N$.
The semi-discrete (mild) problem is then to find $X_N:\gO\times\bT\to V_N$ such that for all $t\in\bT$ there holds $\bP$-a.s.
\be
\label{eq:mild-semi}
X_N(t)=S_N(t)P_NX_0
+\int_0^tS_N(t-s)P_N F(X_N(s)) ds
+\int_0^tS_N(t-s)P_NG(X_N(s))dW(s).
\ee

\subsection{Noise Approximation}
Recall the \KL expansion of $W$ from Equation~\eqref{eq:KL}, where the scalar products $(W(\cdot),e_k)_H$ are real-valued, independent and scaled Brownian motions with variance $\eta_k\ge0$ (the $k$-th eigenvalue of $Q$).
In general, infinitely many of the eigenvalues $\eta_k$ are strictly greater than zero, hence we truncate the series in Eq.~\eqref{eq:KL} after $K\in\bN$ terms to obtain
\bee
W_K(t):=\sum_{k=1}^K (W(t),e_k)_He_k,\quad t\in\bT.
\eee
It can be shown, see for example \cite{BS18a}, that $W_K$ converges to $W$ in mean-square uniformly on $\bT$ with truncation error given by
\bee
\bE\LR{\|W_K(t)-W(t)\|_U^2} = t\sum_{k>K}\eta_k,\quad t\in\bT.
\eee
Combining the semi-discrete mild formulation from~\eqref{eq:mild-semi} with the noise truncation then yields the problem
to find $X_{N,K}:\gO\times\bT\to V_N$ such that for all $t\in\bT$ there holds $\bP$-a.s.
\be
\label{eq:mild-semi-noise}
X_{N,K}(t)=S_N(t)P_NX_0
+\int_0^tS_N(t-s)P_N F(X_{N,K}(s)) ds
+\int_0^tS_N(t-s)P_NG(X_{N,K}(s))dW_K(s).
\ee

\subsection{Time Stepping}
The temporal discretization is based on rational approximations of $S_N$.
Recall that $(-A_N):V_N\to V_N$ is a linear, positive definite, self-adjoint operator and that $N=\dim(V_N)\in\bN$. There exists an $H$-orthonormal eigenbasis $(\widetilde f_1, \dots, \widetilde f_N)\subset V_N$ of eigenfunctions of $(-A_N)$, with corresponding non-decreasing eigenvalues $(\widetilde\gl_1,\dots, \widetilde\gl_N)$ such that $\widetilde\gl_1 > 0$. We denote the spectrum of $(-A_N)$ by $\gs(-A_N)$ and consider a rational function $r$ defined on $\gs(-A_N)$.

Now fix $M\in\bN$ and let $0=t_0<t_1<\dots<t_M=T$ be an equidistant grid of $[0,T]$ with time step $\gD t:=\nicefrac{T}{M}$. Further, let $r(\gD t A_N)$ be a rational approximation of  $S_N(\gD t)=\exp(-\gD t A_N)$, given by
\begin{equation}\label{eq:rational-approx}
	r(\gD t A_N)v = \sum_{n=1}^{N} r(\gD t \widetilde\gl_n) (v, \widetilde f_n)_H\widetilde f_n, \quad v\in H.
\end{equation}
The drift part in~\eqref{eq:mild-semi-noise} is then approximated in each time step by the forward difference
\begin{equation*}
	\int_{t_m}^{t_{m+1}}S_N(t_{m+1}-s)P_NF(X_{N,K}(s))ds
	\approx r(\gD t A_N)P_NF(X_{N,K}(t_m)) \gD t.
\end{equation*}

To introduce the approximation of the stochastic integral, recall that $G':H\to \cL(H, \LHS(\cU, H))$ denotes the Fr\'echet derivative of $G$.
For any $k\in\bN$ such that $\eta_k>0$, we define $w_k:=\eta_k^{-\nicefrac{1}{2}}(W,e_k)_U$, hence $(w_k,k\in\bN)$ is the sequence of independent Brownian motions in the \KL expansion of $W$.
Further, for $m=0,\dots, M-1$ and any stochastic process $\mfW:\gO\times\bT\to \cH$ with $\cH\in\{\bR, U, \cL_1(U)\}$, we denote by $\gD_m \mfW:= \mfW(t_{m+1})-\mfW(t_m)$ the increment with timestep $[t_m,t_{m+1}]$ (we will use in particular $\mfW\in \{W, W_K, w_k\}$).
We employ a truncated Milstein scheme to approximate the stochastic integral in~\eqref{eq:mild-semi-noise} by a first order Taylor expansion of $G$ via
\begin{align*}
	&\int_{t_m}^{t_{m+1}}S_N(t_{m+1}-s)P_NG(X_{N,K}(s))dW_K(s)\\
	\approx
	&\int_{t_m}^{t_{m+1}}S_N(t_{m+1}-s)P_NG(X_{N,K}(t_m)) dW_K(s) \\
	&\quad+
	\int_{t_m}^{t_{m+1}}S_N(t_{m+1}-s)P_N
	\left[G'(X_{N,K}(t_m))\left(\int_{t_m}^{s}S_N(s-r)P_NG(X_{N,K}(r))dW_K(r)\right)\right]\,dW_K(s) \\
	\approx
	&r(\gD tA_N)P_NG(X_{N,K}(t_m))\gD_m W_K  \\
	&\quad+ r(\gD tA_N)P_N\int_{t_m}^{t_{m+1}}
	G'(X_{N,K}(t_m))\left(P_NG(X_{N,K}(t_m))\int_{t_m}^{s}dW_K(r)\right)dW_K(s)\\
	\approx
	&r(\gD tA_N)P_NG(X_{N,K}(t_m))\gD_m W_K  \\
	&\quad+
	\frac{r(\gD tA_N)P_N}{2}
	\sum_{k,l=1}^K
	G'(X_{N,K}(t_m))\left(P_NG(X_{N,K}(t_m))e_l\right)e_k
	\LR{\sqrt{\eta_k\eta_l}\gD_m w_k\gD_m w_l-\gd_{k,l}\eta_k\gD t},
\end{align*}
where $\gd_{k,l}$ is the Kronecker delta. This approximation corresponds to
the truncated Milstein scheme in \cite{giles2014antithetic} for
finite-dimensional SDEs. \add{Moreover, compared to the Milstein scheme for
  SPDEs
  \cite{jentzen:milstein-SPDE,barth:milstein-SPDE,hallern:milstein-SPDE}, the
  truncated Milstein scheme drops the terms which involve iterated integrals
  of the underlying Wiener processes and is thus identical to the Milstein
  scheme for commutative noise.}
Now define for any $s\in [t_m,T]$ the \(\correctSpace(U)\)-valued process
\begin{equation}\label{eq:tensor_W}
	\cW_{m,K}(s):=(W_K(s)-W_K(t_m))\otimes (W_K(s)-W_K(t_m)) - (s-t_m)\sum_{k=1}^K\eta_k\, e_k\otimes e_k,
\end{equation}
and note that $\cW_{m,K}$ is a continuous, square-integrable,
$\correctSpace(U)$-valued martingale on $[t_m,T]$. Further, let $Q\otimes
Q\in\cL(\cL_1(U))$ be given by $Q\otimes Q (\phi\otimes \varphi) = Q \phi\otimes Q\varphi$ for all $\phi\otimes \varphi\in
\cL_1(U)$.
As $W_K(s)-W_K(t_m)$ is Gaussian, there is a $C>0$ such that for all $s,t\in [t_m,T]$ with $t\ge s$ there holds
\begin{equation}\label{eq:W_bracket}
	\dbracket{\cW_{m,K}}{\cW_{m,K}}_t - \dbracket{\cW_{m,K}}{\cW_{m,K}}_s\le C (t-s)^2Q\otimes Q.
\end{equation}
We use the operator-valued processes
$\cW_{m,K}$ to write the truncated correction term in a compact form.

\begin{prop}\label{prop:correction_est}
	Let Assumption~\ref{ass:SPDE} hold and let $\cW_{m,K}$ be defined as
	in~\eqref{eq:tensor_W} for $m=0,\dots, M-1$ and $M\in\bN$, and let further
	$\gD_m w_k:=(W(t_{m+1})-W(t_m),e_k)_U$ for $k\in\bN$. There exists a mapping
	$\cG:H \del{\mapsto}\add{\to} \LHS(\LHS(\cU), H)$, such that for any
	$X\in H$ and $M\in\bN$ there holds
	\be\label{eq:tensor_correction}
	\int_{t_m}^{t_{m+1}}\cG(X)d\cW_{m,K}(s)
	= \frac{1}{2} \sum_{k,l=1}^K
	G'(X)\left(P_N G(X)e_l\right)e_k
	\LR{\sqrt{\eta_k\eta_l} \gD_m w_k\gD_m w_l-\gd_{k,l}\eta_k\gD t}.
	\ee
	Moreover, $\cG$ is Fr\'echet differentiable on $H$ and satisfies the linear growth bounds
	\be\label{eq:tensor_norm}
	\|\cG(X)\|_{\LHS(\LHS(\cU), H)}+\|\cG'(X)\|_{\cL(H, \LHS(\LHS(\cU), H))}
	\le
	C(1+\|X\|_H), \quad X\in H,
	\ee
	and
	\be\label{eq:tensor_norm2}
	\|\cG(X)\|_{\LHS(\LHS(\cU), \dot H^\ga)}
	\le
	C(1+\|X\|_{\dot H^\ga}), \quad X\in \dot H^\ga.
	\ee
\end{prop}
\begin{proof}
	See Appendix~\ref{sec:app1}.
\end{proof}

Based on Proposition~\ref{prop:correction_est}, we obtain the fully discrete problem as to find $Y_0^{N,K}, Y_1^{N,K} \dots, Y_M^{N,K}:\gO\to V_N$ such that $Y_0^{N,K}=P_NX_0$ and for all $m=0,\dots, M-1$ there holds
\be\label{eq:truncated_milstein}
\begin{split}
	Y_{m+1}^{N,K}
	&= r(\gD tA_N)P_N\LR{Y_m^{N,K}
		+ F(Y_m^{N,K})\gD t+G(Y_m^{N,K})\gD_m W_K
		+ \cG(Y_m^{N,K}) \gD_m\cW_{m,K}},
\end{split}
\ee
where we have used~\eqref{eq:tensor_correction} to define the last term in~\eqref{eq:truncated_milstein} as
\bee
\cG(Y_m^{N,K})\gD_m\cW_{m,K}:= \int_{t_m}^{t_{m+1}}\cG(Y_m^{N,K}) d\cW_{m,K}(s), \quad m=0,\dots, M-1.
\eee

The first tree terms on the right hand side of \eqref{eq:truncated_milstein} correspond to an Euler approximation of $X$, the fourth term is the truncated Milstein correction.
We emphasize that \emph{the scheme in \eqref{eq:truncated_milstein} does not
  require the simulation of any iterated stochastic integrals}, and is
therefore straightforward to implement \add{relative to the standard Milstein
  scheme \cite{jentzen:milstein-SPDE}}.
We formulate the following assumption on strong and weak convergence of the fully discrete scheme. \add{In the following, let the standard Euler discretization of~\eqref{ass:SPDE} be given by $\widehat Y_0=P_NX_0$ and
  \be\label{eq:euler-scheme}
  \widehat Y_{m+1}^{N,K} = r(\gD tA_N) \widehat Y_m^{N,K}
  + r(\gD tA_N)P_NF(\widehat Y_m^{N,K})\gD t
  +r(\gD tA_N)P_NG(\widehat Y_m^{N,K})\gD_m W_K,
  \ee
  for $m=0,\dots, M-1$. See, e.g., \cite[Section 3.6]{kruse2014strong}, and contrast
  this standard scheme to the truncated Milstein scheme in \eqref{eq:truncated_milstein}.
}

\begin{assumption}\label{ass:approximation} ~
	\begin{enumerate}[label=(\roman*)]
		\item\label{item:rational} The rational approximation $r$ of $S_N$ is\del{of order $q\in\bN$ and} stable \add{and at least first order}. That is, $r(z)=e^{-z}+\cO(z^{\del{q+1}\add{2}})$ as $z\to 0$, $|r(z)|<1$ for $z>0$ and $\lim_{z\to \infty} r(z)=0$.

		\item\label{item:subspace_approx} Subspace approximation property: Fix $\ga>0$ and let $(V_N, N\in\bN)$ be a sequence of subspaces $V_N\subset V$ such that $\dim(V_N)=N$. There are constants $C, \widetilde \ga>0$, depending on $\ga$ and $d$,
		such that for any $N\in\bN$ and any $v\in\dot{H}^\ga$ there holds
		$$
		\norm{v-P_Nv}_H\le C N^{-\widetilde \ga} \norm{v}_{\dot{H}^\ga},
		\quad\text{and}\quad \norm{A_N^{\nicefrac{\min(\ga, 2)}{2}} P_N v}_H \le C \norm{v}_{\dot{H}^{\min(\ga, 2)}}.
		$$

              \item\label{item:strong} Strong convergence: There are constants
                $C, \widetilde \ga, \gb>0$ such that for $p\in(0,8]$ and all
                discretization parameters $M, N, K\in\bN$ there holds the
                strong error estimate \add{for the standard Euler discretization}
		\begin{equation*}
                  \sup_{0 \leq t \leq T}
                  \|X(t)-\del{Y}\add{\widehat Y}_m^{N,K}\|_{L^p(\gO; H)}\le C \left(M^{-\nicefrac{1}{2}} + N^{-\widetilde \ga} +  K^{-\gb}\right).
		\end{equation*}
	\end{enumerate}
      \end{assumption}

      Note in particular that if the eigenvalues of $Q$ exhibit the decay
      $\eta_j\le C j^{-\gb_0}$ for some $\gb_0>1$, we may choose
      $\gb=\nicefrac{1}{2}(\gb_0-1-\gd)>0$ for an arbitrary small
      $\gd\in(0,\gb_0-1)$ in Item~\ref{item:strong}.

\begin{ex}\label{ex:FEM}
	~
	\begin{enumerate}
		\item Assume the setting in Example~\ref{ex:dotH}, i.e., we consider the stochastic heat equation on a convex domain $\cD$.
		Let $V_N$ be a space of linear finite elements with respect to a regular triangulation of $\cD$ with mesh width $h=\cO(N^{-d})$ for $N\in\bN$.
		Assumption~\ref{ass:approximation} then holds with $\widetilde \ga=\nicefrac{\min(\ga, 2)}{d}$, where $\ga$ is the spatial Sobolev regularity of $X$ as in Assumption~\ref{ass:SPDE}, see e.g. \cite[Chapters 3 and 5]{kruse2014strong}.

		\item For the Dirichlet-Laplacian in Example~\ref{ex:dotH}, we have by Weyl's law that $\gl_n=\cO(n^{\nicefrac{2}{d}})$ as $n\to\infty$.
		For a spectral Galerkin approach with $V_N=\text{span}\set{f_1,\dots, f_N}$, Assumption~\ref{ass:approximation} thus holds with $\widetilde \ga=\nicefrac{\ga}{d}$ and we obtain the stronger relation
		$$
		\norm{A_N^{\nicefrac{\ga}{2}} P_N v}_H \le \norm{v}_{\dot{H}^\ga},
		$$
		for $\ga\ge 2$. However, this will not affect efficiency of our antithetic scheme in Section~\ref{sec:anti-mlmc}, therefore we formulated Item~\ref{item:subspace_approx} in a unified way to encompass spectral Galerkin and finite element approaches.
	\end{enumerate}
\end{ex}

We conclude this section by recording an error estimate on the truncated Milstein approximation.

\begin{thm}\label{thm:stability}
	Let Assumptions~\ref{ass:SPDE} and~\ref{ass:approximation}~\ref{item:rational} hold, and denote by $Y_\cdot^{N,K}:\{0,\dots,M\}\times\gO\to H$ the truncated Milstein approximation in~\eqref{eq:truncated_milstein}.
	\begin{enumerate}
		\item For any $p\in[2,8]$ there is a constant $C>0$, independent of $M, N$ and $K$, such that
		\begin{equation}
			\max_{m=0,\dots, M}\E{\norm{Y_m^{N,K}}_H^p}\le C(1 + \E{\norm{X_0}_H^p}) <\infty.
		\end{equation}

		\item \rdel{Let the standard Euler discretization of~\eqref{ass:SPDE} be given by $\widehat Y_0=P_NX_0$ and
		\bee
		\widehat Y_{m+1}^{N,K} = r(\gD tA_N) \widehat Y_m^{N,K}
		+ r(\gD tA_N)P_NF(\widehat Y_m^{N,K})\gD t
		+r(\gD tA_N)P_NG(\widehat Y_m^{N,K})\gD_m W_K,
		\eee
		for $m=0,\dots, M-1$.}
              For the truncated Milstein scheme in
              \eqref{eq:truncated_milstein} and the Euler scheme in
              \eqref{eq:euler-scheme}, and for any $p\in(0,4]$, there exists a
              constant $C>0$, independent of $M, N$ and $K$, such that
		\begin{equation*}
			\max_{m=0,\dots, M}
			\E{\|Y_m^{N,K}-\widehat Y_m^{N,K}\|_H^p}^{1/p}\le C M^{-\nicefrac{1}{2}}.
		\end{equation*}
	\end{enumerate}
\end{thm}

\begin{proof}
	See Appendix~\ref{sec:app1}.
\end{proof}

\add{
      \begin{cor}\label{cor:strong-with-milstein}
        Let Assumptions~\ref{ass:SPDE}, ~\ref{ass:approximation}~\ref{item:rational}
        and \ref{ass:approximation}~\ref{item:strong} hold, then there are constants $C, \widetilde \ga, \gb>0$ such that for $p\in(0,8]$ and all discretization parameters $M, N, K\in\bN$ there holds the strong
        error estimate
        \begin{equation*}
          \max_{m=0,\ldots, M}
          \|X(m \Delta t)-Y_m^{N,K}
          \|_{L^p(\gO; H)}\le C \left(M^{-\nicefrac{1}{2}} + N^{-\widetilde \ga} +  K^{-\gb}\right).
        \end{equation*}
      \end{cor}
      \begin{proof}
        This follows by a simple application of the triangle inequality,
        Assumption~\ref{ass:approximation}-\ref{item:strong} and
        Theorem~\ref{thm:stability}.
      \end{proof}
    }

    \del{Item~\ref{item:strong}}
    \add{Corollary~\ref{cor:strong-with-milstein}} essentially states that the truncated Milstein term does neither increase, nor spoil the strong rates of convergence, as compared to the standard Euler scheme. \del{We justify this in Theorem~\ref{thm:stability} below.}
    However, the Milstein correction accelerates variance decay for an antithetic coupling in the MLMC estimator.
\section{Antithetic \MMLMC-Milstein Scheme}
\label{sec:anti-mlmc}

To construct an antithetic estimator, we consider coupled "coarse" and "fine" approximations of $X$ given by the truncated Milstein scheme in~\eqref{eq:truncated_milstein}, with refinement parameters $M, N$ and $K$ adjusted accordingly.
First, let $Y_\cdot^c:=Y_\cdot^{N,K}$ be the coarse step discretization with a fixed time step $\gD t=\nicefrac{T}{M}$ and fixed $N,K\in\bN$ as in~\eqref{eq:truncated_milstein}.
That is, $Y_m^c$ is given by $Y_0^c:=P_NY_0$ and
\be\label{eq:coarse}
\begin{split}
	Y_{m+1}^c
	&= r(\gD tA_N)P_N\LR{ Y_m^c
		+ F(Y_m^c)\gD t+G(Y_m^c)\gD_m W_K
		+ \cG(Y_m^c) \gD_m\cW_{m,K}},\quad m=0,\dots, M-1.
\end{split}
\ee
For the corresponding fine approximation, let $\mfd t:=\nicefrac{\gD t}{2}$ and set a cutoff index $K_f\in\bN$ such that $K_f\ge K$ for the noise approximation. We define the fine increments
\begin{alignat*}{2}
	&\mfd_m W_{K_f}:= W_{K_f}(t_m+\mfd t) - W_{K_f}(t_m),\quad
	&&\mfd_{m+1/2} W_{K_f} := W_{K_f}(t_{m+1}) - W_{K_f}(t_m+\mfd t) \\
	&\mfd_m\cW_{m,K_f} := \cW_{m,K_f}(t_m+\mfd t), \quad
	&&\mfd_{m+1/2}\cW_{m,K_f} := \mfd_{m+1/2} W_{K_f}\otimes \mfd_{m+1/2} W_{K_f}
	-\mfd t\sum_{k=1}^{K_f} \eta_k e_k\otimes e_k.
\end{alignat*}
Note that $\gD_m \cW_{m,K_f} \neq \mfd_{m+1/2}\cW_{m,K_f} + \mfd_m\cW_{m,K_f}$, but there holds
\begin{equation}\label{eq:tensor_increment}
	\gD_m\cW_{m,K_f}= \mfd_{m+1/2}\cW_{m,K_f} + \mfd_m\cW_{m,K_f}
	+ \mfd_{m+1/2}W_{K_f}\otimes \mfd_mW_{K_f}
	+ \mfd_mW_{K_f}\otimes \mfd_{m+1/2}W_{K_f}.
\end{equation}
We further consider a finite dimensional subspace $V_{N_f}\subset V$ with $\dim(V_{N_f})=N_f$ such that $N_f\ge N=\dim(V_N)$.
The corresponding discrete operator is denoted by $A_{N_f}:V_{N_f}\to V_{N_f}$, its associated semigroup by $S_{N_f}$, and $P_{N_f}:H\to V_{N_f}$ is the $H$-orthogonal projection onto $V_{N_f}$.
Finally, we set a cutoff index $K_f\in\bN$ such that $K_f\ge K$ for the noise approximation on the fine scale.

The \emph{fine step discretization} $Y^f_\cdot:\gO\times\{0,\nicefrac{1}{2},1,\dots,M-\nicefrac{1}{2}, M\}\to V_{N_f}$ is then given by $Y_0^f:=P_{N_f} X_0$,
\be\label{eq:fine1}
\begin{split}
	Y_{m+1/2}^f
	&= r(\mfd tA_{N_f})P_{N_f}
	\LR{ Y_m^f+F(Y_m^f)\mfd t+G(Y_m^f)\mfd_m W_{K_f} + \cG(Y_m^f)\mfd_m\cW_{m,K_f}},
\end{split}
\ee
and
\be\label{eq:fine2}
\begin{split}
	Y_{m+1}^f
	&= r(\mfd tA_{N_f})P_{N_f}
	\LR{ Y_{m+1/2}^f+F(Y_{m+1/2}^f)\mfd t+G(Y_{m+1/2}^f)\mfd_{m+1/2} W_{K_f} + \cG(Y_{m+1/2}^f)\mfd_{m+1/2}\cW_{m,K_f}}.
\end{split}
\ee
The \emph{antithetic counterpart} $Y^a_\cdot:\gO\times\{0,\nicefrac{1}{2},1,\dots,M-\nicefrac{1}{2}, M\}\to V_{N_f}$ to $Y^f_\cdot$ is defined via $Y_0^a:=P_{N_f} X_0$,
\be\label{eq:anti1}
\begin{split}
	Y_{m+1/2}^a
	&= r(\mfd tA_{N_f})P_{N_f}
	\LR{ Y_m^a+F(Y_m^a)\mfd t+G(Y_m^a)\mfd_{m+1/2} W_{K_f} + \cG(Y_m^a)\mfd_{m+1/2}\cW_{m,K_f}},
\end{split}
\ee
and
\be\label{eq:anti2}
\begin{split}
	Y_{m+1}^a
	&= r(\mfd tA_{N_f})P_{N_f}
	\LR{ Y_{m+1/2}^a+F(Y_{m+1/2}^a)\mfd t+G(Y_{m+1/2}^a)\mfd_m W_{K_f} + \cG(Y_{m+1/2}^a)\mfd_m\cW_{m,K_f}}.
\end{split}
\ee

The foundation of our MLMC Milstein approach is to show that the difference of the \emph{antithetic average}
\begin{equation}\label{eq:antiavg}
	\ol Y_m:=\frac{Y_m^f+Y_m^a}{2}, \quad m=0,\dots, M,
\end{equation}
to the coarse scale $Y_m^c$ approximations exhibits a rapid decay in mean-square.
This property is established in our main result:

\begin{thm}
	\label{thm:MainRes}
	Let Assumptions~\ref{ass:SPDE} and~\ref{ass:approximation} hold for some $\ga\ge1$,
	and let $M, N_f, N, K_f, K\in\bN$ be such that $N_f\ge N$ and $K_f\ge  K$.
	Further, let $Y_\cdot^c$ be as in~\eqref{eq:coarse} and let $\ol Y_\cdot$ be the antithetic average of the fine approximations as in~\eqref{eq:antiavg}.
	Then, there is a constant $C>0$, independent of $M, N, $ and $K$ such that
	\begin{equation}\label{eq:var-decay}
		\max_{m=0,\dots, M}\E{\norm{\ol Y_m-Y_m^c}_H^2}
		\le C\left(M^{-\min(\ga,2)} + N^{-2\widetilde \ga} + K^{-2\gb}\right).
	\end{equation}
\end{thm}

\begin{proof}
	See Appendix~\ref{sec:app2}.
\end{proof}

\begin{rem}
  For the truncated Milstein scheme without antithetic correction,
  \add{Corollary~\ref{cor:strong-with-milstein} implies}
	\begin{equation}\label{eq:var-decay-EM}
		\max_{m=0,\dots, M}\E{\norm{Y_m^f-Y_m^c}_H^2}
		\le C\left(M^{-1} + N^{-2\widetilde \ga} + K^{-2\gb} \right),
	\end{equation}
	and therefore a slower variance decay with respect to the time step $\gD t = T/M$.
\end{rem}

\section{Multilevel Monte Carlo Approximation}
\label{sec:MLMC}

Let $Z:\gO\to\bR$ be a real-valued, integrable random variable, and let $(Z^{(i)}, i\in\bN)$ be a sequence of independent copies of $Z$. For any finite number of samples $\mfN\in\bN$ we define the singlelevel \MC\, estimator of $\E{Z}$ by
\begin{equation}
	E_\mfN(Z):=\frac{1}{\mfN}\sum_{i=1}^\mfN Z^{(i)}.
\end{equation}

We aim to estimate $\E{\Psi(X_T)}$ for a given functional $\Psi:H\to\bR$ by \MLMC\,(MLMC) methods.
To this end, let $M_0, L\in\bN$ and let $M_\ell = M_02^\ell$ for $\ell=1,\dots, L$.
Based on Theorem~\ref{thm:MainRes}, we balance the error contributions in~\eqref{eq:var-decay} on all levels by setting the remaining approximation parameters as
\begin{equation}\label{eq:balance}
  N_\ell :=
  \lceil M_\ell^{\nicefrac{\min(\ga,2)}{2\widetilde \ga}} \rceil \quad\text{and}\quad
	K_\ell := \lceil M_\ell^{\nicefrac{\min(\ga,2)}{2\gb}} \rceil, \quad \ell=1,\dots, L.
\end{equation}

We denote for $\ell=2,\dots, L$ by $Y^{c,\ell-1}$ the coarse step approximation in~\eqref{eq:coarse} with discretization parameters given by $M_{\ell-1}, N_{\ell-1}, K_{\ell-1}$.
For $\ell=1,\dots, L$, we let denote by $Y^{f,\ell}, Y^{a,\ell}$ the
fine step discretization and its antithetic counterpart, respectively, both with discretization parameters $M_{f} = M_\ell=2M_{\ell-1}$ and $N_f=N_\ell, K_f=K_\ell$.
Furthermore, we define $\ol Y^\ell=\nicefrac{1}{2}(Y^{f,\ell}+ Y^{a,\ell})$,
\begin{equation}\label{eq:qoi}
	\Psi_0^c:=0,
	\quad,
	\Psi_\ell^c:=\Psi(Y_{M_\ell}^{c,\ell}), \quad\text{and}\quad
	\ol \Psi_\ell:=\frac{\Psi(Y_{M_\ell}^{f,\ell})+\Psi(Y_{M_\ell}^{a,\ell})}{2}, \quad\text{for $\ell=1,\dots,L$.}
\end{equation}

We introduce the \emph{antithetic \MLMC\, estimator} as
\begin{equation}\label{eq:mlmc-est}
	E_L^{\text ML}(\Psi):=\sum_{\ell=1}^L E_{\mfN_\ell}(\ol \Psi_\ell - \Psi_{\ell-1}^c),
\end{equation}
where $\mfN_1,\dots, \mfN_L\in\bN$ are level-dependent numbers of samples.
Since $Y_M^{f,\ell}\stackrel{d}{=}Y_M^{a,\ell}$, it holds that
\begin{equation*}
	\E{E_L^{\text ML}(\Psi)}=\E{\Psi(Y_{M_L}^{f,L})}.
\end{equation*}

To analyze the mean-squared error (MSE) and computational complexity of the estimator in~\eqref{eq:mlmc-est}, we formulate the following assumptions on the sample complexity and the weak error.

\begin{assumption}\label{ass:mlmc}
	\begin{enumerate}[label=(\roman*)]
		For fixed $M_0\in\bN$ and any $\ell\in\bN$, let $M_\ell=M_02^\ell$ and $N_\ell, K_\ell\in\bN$ be as in~\eqref{eq:balance}.
		~
		\item\label{item:comlexity} Sample complexity:
		Denote by $\cC_\ell$ the cost of generating one sample of $\ol\Psi^\ell$ on any a given refinement level $\ell\in\bN$.
		There are constants $C>0$ and $\gg>0$ such that for any $\ell\in\bN$ there holds
		\begin{equation*}
			\cC_\ell\le C M_\ell^{1+\gg}.
		\end{equation*}

		\item\label{item:weak} Weak convergence:
		Let $\widetilde \ga$ and $\gb$ be as in Assumption~\ref{ass:approximation}~\ref{item:strong}, let $\Psi:H\to \bR$ be Fr\'echet differentiable with bounded derivative, and let $\gd\in(0,1)$ be arbitrary small.
		There is a constant $C=C(\Psi, \gd)>0$ such that for $\ell\in\bN$ there holds
		\begin{equation*}
			\abs{\bE(\Psi(X(T)))-\bE(\Psi(Y_{M_\ell}^{f,\ell}))}
			\le C M_\ell^{-1+\gd}.
		\end{equation*}

	\end{enumerate}
\end{assumption}

\begin{rem}
	\label{rem:weak_approx}
	The parameter $\gg$ in Item~\ref{item:comlexity} essentially depends on
        the cost of evaluating $G$ \add{and its Fr\'echet derivative \(G'\),
          or equivalently \(\cG\), in~\eqref{eq:truncated_milstein}}. In case
        there is some sparsity to exploit in $G$, the cost of one evaluation
        may be as low as $\cO(\max(N_\ell,K_\ell))$, in which
        case~\eqref{eq:balance} yields that $\gg =
        \nicefrac{1}{2}\cdot\nicefrac{\min(\ga, 2)}{\min(\widetilde\ga, \gb)}$,
        see for instance the numerical example in Section~\ref{sec:numerics}.
        On the other hand, the cost of one evaluation may be as large as
        $\cO(N_\ell^2K_\ell^2)$, if evaluating $G$ \add{and \(G'\)} entails full matrices and nested
        summations in the discretization scheme, which makes each sample
        significantly more expensive.

	Assumption~\ref{ass:mlmc}~\ref{item:weak} on the weak approximation
        error is natural, one often recovers (almost) twice the strong rates
        \add{of an Euler scheme} for semi-linear, parabolic SPDEs, see e.g.
        \cite[Theorem 5.12]{kruse2014strong} or \cite[Theorem 4.5]{KLS15}. In
        other words, the weak error with respect to $N_\ell$ and $K_\ell$ is of
        order $\cO(N_\ell^{-2\widetilde\ga+\gd}+K_\ell^{-2\gb+\gd})$ for any
        arbitrary small $\gd>0$, and the balancing in~\eqref{eq:balance}
        yields with $\ga\ge 1$ that
	\begin{equation*}
		\abs{\bE(\Psi(X(T)))-\bE(\Psi(Y_{M_\ell}^{f,\ell}))}
		=
		\cO(M_\ell^{-1+\gd} + N_\ell^{-2\widetilde\ga+\gd}+K_\ell^{-2\gb+\gd})
		=
		\cO(M_\ell^{-1+\gd}).
	\end{equation*}

\end{rem}

\begin{thm}\label{thm:mlmc-comp}
	Let Assumptions~\ref{ass:SPDE},~\ref{ass:approximation} and~\ref{ass:mlmc} hold \add{and let \(\Psi\in C_b^2(H,\bR)\)}.
	For any $\eps\in(0,e^{-1})$, there exists an antithetic \MLMC-Milstein estimator \del{$E_L(\Psi(Y_M))$}
        \add{\(E_L^{\text ML}(\Psi)\)} of $\bE(\Psi(X(T)))$ such that there holds
	\begin{equation*}
		\bE\LR{\abs{E_L^{\text ML}(\Psi)-\bE(\Psi(X(T)))}^2}<\eps^2,
	\end{equation*}
	and the computational complexity $\cC_{\rm ML}$ of \del{$E_L(\Psi(Y_M))$}\add{\(E_L^{\text ML}(\Psi)\)} is bounded by
	\begin{equation}\label{eq:mlmc-complexity}
		\cC_{\rm ML}\le
		\begin{cases}
			C \eps^{-2}, &\quad \min(\ga, 2)>1+\gg,\\
			C \eps^{-2}|\log(\eps)|^2, &\quad \min(\ga, 2)=1+\gg,\\
			C \eps^{-2-\frac{1+\gg-\min(\ga, 2)}{1-\gd}}, &\quad \min(\ga, 2)<1+\gg.
		\end{cases}
	\end{equation}
\end{thm}

\begin{proof}
	See Appendix~\ref{sec:app3}.
\end{proof}

\begin{ex}
	To show that all three cases in~\eqref{eq:mlmc-complexity} are conceivable, recall Example~\ref{ex:FEM}, where $\widetilde \ga=\nicefrac{\min(\ga, 2)}{d}$ (FEM) or  $\widetilde \ga=\nicefrac{\ga}{d}$ (spectral Galerkin method).
	Assuming for simplicity that $\ga\in[1,2]$, $\widetilde \ga=\nicefrac{\ga}{d}$, and that $G$ may be  evaluated with complexity $\cO(\max(N_\ell,K_\ell))$. The error balancing~\eqref{eq:balance} then yields
	\begin{equation*}
		\max(N_\ell,K_\ell) =
		M_\ell^{\ga\max(\nicefrac{d}{2\ga},\nicefrac{1}{2\gb})},
	\end{equation*}
	and thus Item~\ref{item:comlexity} holds with
	\begin{equation*}
		\gg = \ga\max(\nicefrac{d}{2\ga},\nicefrac{1}{2\gb}) < \ga - 1
		\quad \Longleftrightarrow \quad
		\max(\nicefrac{d}{2}, \nicefrac{\ga}{2\gb}) < \ga - 1.
	\end{equation*}
\end{ex}

\section{Numerics}
\label{sec:numerics}

Let {$\cD = [0,1]^d$}{} for $d\in\{1,2,3\}$, let
$H:=L^2(\cD)$ and denote by $A:=\laplace $ the Laplace-operator with
homogeneous Dirichlet boundary conditions. We further assume that $U=H$ and
denote by $(e_k, k\in\bN)$ and $(\gl_k, k\in\bN)$ the eigenfunctions and
eigenvalues of $-A$, respectively. By Weyl's law
$\gl_k=\cO(k^{\nicefrac{2}{d}})$ for $k\in\bN$ and for rectangular and circular
domains the precise eigenfunctions and eigenvalues of \(\laplace\) are given in closed form,
see e.g. \cite[Section
3]{grebenkov2013geometrical}. %
We consider the stochastic heat equation given by
\begin{equation}\label{eq:SHE}
	dX(t) = \laplace X(t) dt + G(X(t))dW(t),\quad t\in[0,1], \quad X(0) = X_0,
\end{equation}
for $X_0\in L^8(\gO; \dot H^2)$. The driving noise is modeled by a $Q$-Wiener
process $W:\gO\times [0,1] \to H$ with covariance operator $Q:=((-\laplace)^{-s})$ for a smoothness
parameter $s>0$. Since $\gl_k=\cO(k^{\nicefrac{2}{d}})$ for $k\in\bN$, $Q$ is
trace-class for $s>\nicefrac{d}{2}$, in which case $W$ admits
the \KL expansion
\begin{equation}\label{eq:KL-W}
	W(t) :=\sum_{k\in\bN} \eta_{k}^{1/2} w_k(t) e_k ,
\end{equation}
where \(\eta_{k} = \lambda_{k}^{-s}\) and $w_1, w_2, \dots$ are independent
one-dimensional Brownian motions. Hence, truncating the
expansion~\eqref{eq:KL-W} after $K$ terms yields an error of order
$\cO(K^{1-2s/d})$ with respect to $\norm{\cdot}_H^2$, uniform in $[0,T]$, and
implies that Assumption \ref{ass:approximation} holds with
\(\beta=\LR{2s/d-1}/2\). Alternatively, we could define the diffusion part of
Equation~\eqref{eq:SHE} as $\widehat G(X_t)d\widehat W_t$, where $\widehat W$
is Gaussian white noise (i.e. has covariance operator $\widehat Q = \cI$) and
$\widehat G(v)e_k := \eta_{k}^{1/2} G(v)e_{k}$ for all $v\in H$ and $k\in\bN$.
By \cite[Theorem 2.27]{kruse2014strong} it then follows that
Equation~\eqref{eq:SHE} admits a unique mild solution $X$ such that $X(t)\in L^8(\gO; \dot
H^{\ga})$ for $\ga\in[1, \min(1+s, 2)]$ and all $t\in[0,1]$.
We fix the diffusion coefficient $G$ to be
linear and given by
\begin{equation*}
	G(v)u := \sum_{k=1}^\infty (v,e_k)_He_{k+1}(u, \sqrt{\eta_{k+1}}e_{k+1})_\cU +(g,e_k)_He_k(u, \sqrt{\eta_k}e_k)_\cU
\end{equation*}
for all $v\in H, u \in \cU=Q^{\nicefrac{1}{2}}H$ and some fixed \(g \in H\). We use
a spectral Galerkin approach and expand \(Y^{c}, Y^{f}\) and \(Y^{a}\) in the
same basis, for example
\[
Y_{m}^{c} = \sum_{k=1}^{N} y_{m, k}^{c} \, e_{k}
\]
for \(\lbrace y_{m,k}^{c} \rbrace_{k=1}^{N} \in \bR^{N}\) and \(m=1,\ldots, M\). Recall from
Example~\ref{ex:FEM} that Assumption~\ref{ass:approximation} then holds with
$\widetilde \ga=\nicefrac{\ga}{d}$. The scheme in \eqref{eq:coarse} simplifies
to
\[
\begin{aligned}
	y_{m+1, 1}^{c} &= r(\gD t \, \gl_1) \, \LR{y_{m,1} + (g,e_1)_H  \gD_m w_1} \\
	y_{m+1, 2}^{c} &= r(\gD t \, \gl_2) \LR{\, y_{m,2} +
		\LR{
			\, y_{m,1}
			+(g,e_2)_H  }\gD_m w_2 +
		\frac{1}{2}
		(g, e_{1})_H
		\, \gD_mw_2\gD_m w_{1}} .\\
	y_{m+1, k}^{c} &= r(\gD t \, \gl_k) \LR{\, y_{m,k} + \chi_{k \leq K} \,
		\LR{
			\, y_{m,k-1}
			+(g,e_k)_H  }\gD_m w_k}   \\
	&\quad+ \chi_{k \leq K}
	\frac{r(\gD t \gl_{k})}{2}
	\LR{y_{m,k-2} + (g, e_{k-1})_H}
	\, \gD_mw_k\gD_m w_{k-1} .
\end{aligned}
\]
for \(K\geq 2\), \(k = 3, \ldots, N\), and for \(\chi\) being the characteristic
function. The schemes in \eqref{eq:fine1} and \eqref{eq:fine2} for \(\lbrace{y_{m,
    k}^{f}}\rbrace_{k=1}^{N_{f}}\) and \eqref{eq:anti1} and \eqref{eq:anti2} for
\(\lbrace{y_{m, k}^{a}}\rbrace_{k=1}^{N_{f}}\) simplify similarly. The cost of evaluating
the previous scheme is \(\cO(N)\) for every time step since only \( \mathcal O
(\min(N, K))\) independent Brownian increments \( \gD_m w_{1}, \ldots, \gD_m w_{N}
\) are needed. \add{We choose the rational approximation
  \(r(x)=(1-x/2)/(1+x/2)\)}. We \add{also} set \((g, e_{k})_{H} = k^{-1/2 -
  \varepsilon}\) and \((X_{0}, e_{k})_{H} = k^{-1/2-2/d-\varepsilon}\) for \(k \in \bN\) and some
\(\varepsilon >0\). Note that with this choice {\(G(v)\in\LHS(\cU, H)\) for all $v\in H$}{}
and \(X_{0} \in L^8(\gO; \dot H^2)\). In Figure \ref{fig:l2error-N-plots}, we
fix \(M\) and \(K\) to some sufficiently large values and plot estimates of
the difference \(\max_{m} \|Y_{m}^{N, K} - Y_{m}^{\lceil \sqrt{2} N \rceil, K}\|_{L^{2}(\Omega;
  H)}\) for several values of \(N\). The plot verifies the convergence order
with respect to \(N\) in Assumption \ref{ass:approximation} as
\(\mathcal{O}(N^{-\min(s+1,2)/d})\). %
Next, we choose \(N\) and \(K\) in terms of \(M\) as in \eqref{eq:balance}.
In this case the cost per sample is \(\mathcal O(M^{1+\gg})\) where
\begin{equation*}
  \gg = \max(\nicefrac{d}{2},\nicefrac{\ga}{2\gb}) =
  d\, \max(\nicefrac{1}{2},\nicefrac{\min(1+s, 2)}{(2s - d)}).
\end{equation*}
We plot in Figure \ref{fig:variance-M-plots} estimates of the left-hand sides
of \eqref{eq:var-decay} and \eqref{eq:var-decay-EM} which verifies the claim
of Theorem \ref{thm:MainRes} and the improved convergence order of the
variance for the antithetic estimator. In particular, the variance convergence
order is \(\mathcal{O}(M^{-\min{(\ga, 2)}}) = \mathcal{O}(M^{-\min(s+1, 2)})\) for the antithetic
estimator and \(\mathcal{O}(M^{-1})\) for the truncated Milstein scheme without
antithetic correction. Both figures also clearly showcase the reduced
convergence rates in terms of \(N\) and \(M\) when \(d=1\) and as the
smoothness parameter \(s\) decreases.

\pgfmathdeclarefunction{fitRefLine}{4}{%
	\pgfmathparse{#4 * (#1/#3)^(#2)}}

\pgfplotsset{
  every axis title/.append style={yshift=-1.5ex},
	every axis legend/.style={
		cells={anchor=center},
		fill opacity=0.8, draw opacity=1, text opacity=1,
                draw=none,
                anchor=south west,
		at={(0.03,0.03)},
                legend cell align=left,
		inner xsep=3pt,inner ysep=2pt,
		nodes={inner sep=2pt,text depth=0.15em},
	},
	scale=0.8
}

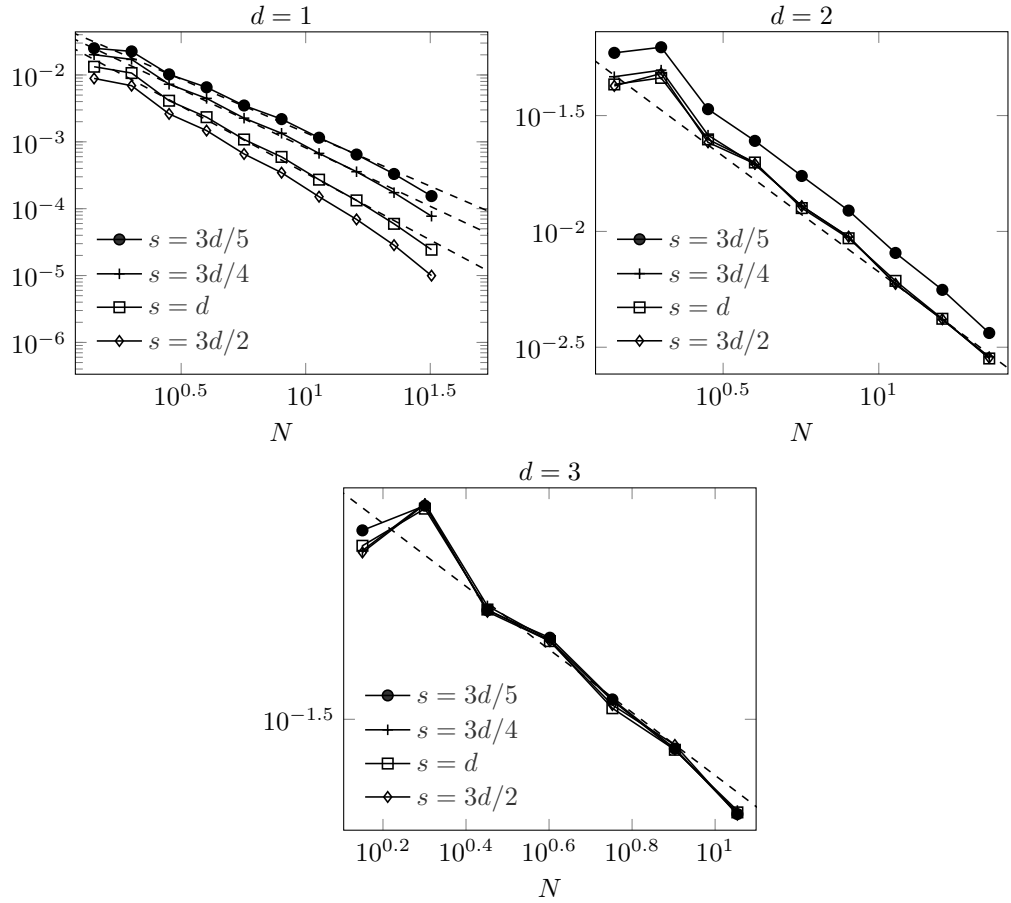
\begin{figure}
	\centering
\begin{tikzpicture}
  \begin{axis}[
    title={\(d=1\)},
    xlabel={\(N\)}, %
    xmin=1.18920711500272, xmax=53.8173705762377,
    ymin=3.34537406705307e-07, ymax=0.0441079926868598,
    xmode=log, ymode=log, log basis x={10}, log basis y={10}
    ]

    \addplot [semithick, black, mark=*, mark size=2]
    table {%
      1.4142135623731 0.0250977330240282
      2 0.0225120373921814
      2.82842712474619 0.0102064805528282
      4 0.00654079716441324
      5.65685424949238 0.00349424389369839
      8 0.0021861156918972
      11.3137084989848 0.00115264089941295
      16 0.000645267020005612
      22.6274169979695 0.000330600386282554
      32 0.000154402781273215
    };
    \addlegendentry{\(s=3d/5\)}
    \addplot [semithick, black, mark=+, mark size=2]
    table {%
      1.4142135623731 0.0200806205508922
      2 0.0171057631066643
      2.82842712474619 0.00726636450489942
      4 0.00446255357732378
      5.65685424949238 0.00226125315477691
      8 0.00134239729270296
      11.3137084989848 0.000671362048934147
      16 0.000357223658752842
      22.6274169979695 0.000173715807963415
      32 7.72804355538637e-05
    };
    \addlegendentry{\(s=3d/4\)}
    \addplot [semithick, black, mark=square, mark size=2]
    table {%
      1.4142135623731 0.01328030282913
      2 0.0106830367902493
      2.82842712474619 0.00411974028280455
      4 0.00233680792362394
      5.65685424949238 0.00108336087003239
      8 0.000594959618398975
      11.3137084989848 0.000271762079206406
      16 0.000133036376394028
      22.6274169979695 5.93922004747208e-05
      32 2.43316959450197e-05
    };
    \addlegendentry{\(s=d\)}
    \addplot [semithick, black, mark=diamond, mark size=2]
    table {%
      1.4142135623731 0.0088726955517452
      2 0.0069146492781663
      2.82842712474619 0.00261845665244441
      4 0.00147312346456528
      5.65685424949238 0.000658451890290235
      8 0.00034539881523812
      11.3137084989848 0.000151064339651146
      16 6.91587228276384e-05
      22.6274169979695 2.84704002539484e-05
      32 9.96772562248637e-06
    };
    \addlegendentry{\(s=3d/2\)}

    \addplot [semithick, black, dashed, domain=1:55]
    {fitRefLine(x, -2, 16, 0.000133036376394028)};

    \addplot [semithick, black, dashed, domain=1:55]
    {fitRefLine(x, -1.75,16, 0.000357223658752842)};

    \addplot [semithick, black, dashed, domain=1:55]
    {fitRefLine(x, -1.6, 16, 0.000645267020005612)};

  \end{axis}
\end{tikzpicture}
\begin{tikzpicture}
  \begin{axis}[
    title={\(d=2\)},
    xlabel={\(N\)}, %
    xmin=1.23114441334492, xmax=25.9920766833995,
    ymin=0.00242224068844143, ymax=0.072820879952109,
    xmode=log, ymode=log, log basis x={10}, log basis y={10}
    ]

    \addplot [semithick, black, mark=*, mark size=2]
    table {%
      1.4142135623731 0.0589970567461984
      2 0.0623838751345702
      2.82842712474619 0.0337059350159465
      4 0.0245938343846685
      5.65685424949238 0.0173523756809095
      8 0.0122994973635307
      11.3137084989848 0.0080726720467484
      16 0.00559257697080515
      22.6274169979695 0.0036435550628823
    };
    \addlegendentry{\(s=3d/5\)}
    \addplot [semithick, black, mark=+, mark size=2]
    table {%
      1.4142135623731 0.046536167647093
      2 0.049687933090109
      2.82842712474619 0.0260276135978104
      4 0.0194299065728082
      5.65685424949238 0.012876347353814
      8 0.0094502087206487
      11.3137084989848 0.00592954487553309
      16 0.00415688572681867
      22.6274169979695 0.00287506001728794
    };
    \addlegendentry{\(s=3d/4\)}
    \addplot [semithick, black, mark=square, mark size=2]
    table {%
      1.4142135623731 0.0430987479130129
      2 0.0460009247590751
      2.82842712474619 0.0249623766000225
      4 0.0198401525484907
      5.65685424949238 0.0126098735758327
      8 0.00934108906193912
      11.3137084989848 0.00611621568389463
      16 0.00419667173837899
      22.6274169979695 0.00282748864201866
    };
    \addlegendentry{\(s=d\)}
    \addplot [semithick, black, mark=diamond, mark size=2]
    table {%
      1.4142135623731 0.0424620614470465
      2 0.0480325165841124
      2.82842712474619 0.0242399145743939
      4 0.0195952553439117
      5.65685424949238 0.0128563829408159
      8 0.00949532778189472
      11.3137084989848 0.00593460222784008
      16 0.0041752553123006
      22.6274169979695 0.00285077590985017
    };
    \addlegendentry{\(s=3d/2\)}

    \addplot [semithick, black, dashed, domain=1:55]
    {fitRefLine(x, -1, 16, 0.0041752553123006)};

  \end{axis}
\end{tikzpicture}

\begin{tikzpicture}
  \begin{axis}[
    title={\(d=3\)},
    xlabel={\(N\)}, %
    xmin=1.27456062731926, xmax=12.553345566348,
    ymin=0.0184242505271278, ymax=0.0976336922349015,
    xmode=log, ymode=log, log basis x={10}, log basis y={10}
    ]

    \addplot [semithick, black, mark=*, mark size=2]
    table {%
      1.4142135623731 0.0793815990750634
      2 0.089516342953107
      2.82842712474619 0.0537667002605838
      4 0.0470474037791559
      5.65685424949238 0.0348349262077071
      8 0.027410022221387
      11.3137084989848 0.0199309215371556
    };
    \addlegendentry{\(s=3d/5\)}
    \addplot [semithick, black, mark=+, mark size=2]
    table {%
      1.4142135623731 0.0719289397722225
      2 0.0905067733944391
      2.82842712474619 0.0549345107035034
      4 0.0464833485821024
      5.65685424949238 0.0346363622568602
      8 0.0272021140518346
      11.3137084989848 0.0202699733904351
    };
    \addlegendentry{\(s=3d/4\)}
    \addplot [semithick, black, mark=square, mark size=2]
    table {%
      1.4142135623731 0.0736525311494264
      2 0.0882224096317956
      2.82842712474619 0.05397408947376
      4 0.0462824533358153
      5.65685424949238 0.0333674004824878
      8 0.0272559825887913
      11.3137084989848 0.0200909357907273
    };
    \addlegendentry{\(s=d\)}
    \addplot [semithick, black, mark=diamond, mark size=2]
    table {%
      1.4142135623731 0.0712833910479774
      2 0.0900899031759922
      2.82842712474619 0.0533882045995173
      4 0.0462752725756581
      5.65685424949238 0.0339084086601254
      8 0.0279098766584302
      11.3137084989848 0.0198750605966784
    };
    \addlegendentry{\(s=3d/2\)}

    \addplot [semithick, black, dashed, domain=1:55]
    {fitRefLine(x, -2/3, 8, 0.0279098766584302)};
  \end{axis}
\end{tikzpicture}
  \caption{Estimates of the \(L^{2}(\Omega ; H)\) difference \(\max_{m} \|Y_{m}^{N,K} - Y_{m}^{\lceil \sqrt{2} N \rceil, K}\|_{L^{2}(\Omega; H)}\) for the numerical
    example in Section~\ref{sec:numerics} and when using the Galerkin method
    for different number of terms, \(N\), in the spatial approximation. The
    estimates were obtained using Monte Carlo sampling with at least 4000
    samples. The dashed reference lines are \(\propto N^{-\min(1+s, 2)/d}\) and
    verify the assumed convergence rates in Assumption
    \ref{ass:approximation}.}
	\label{fig:l2error-N-plots}
\end{figure}

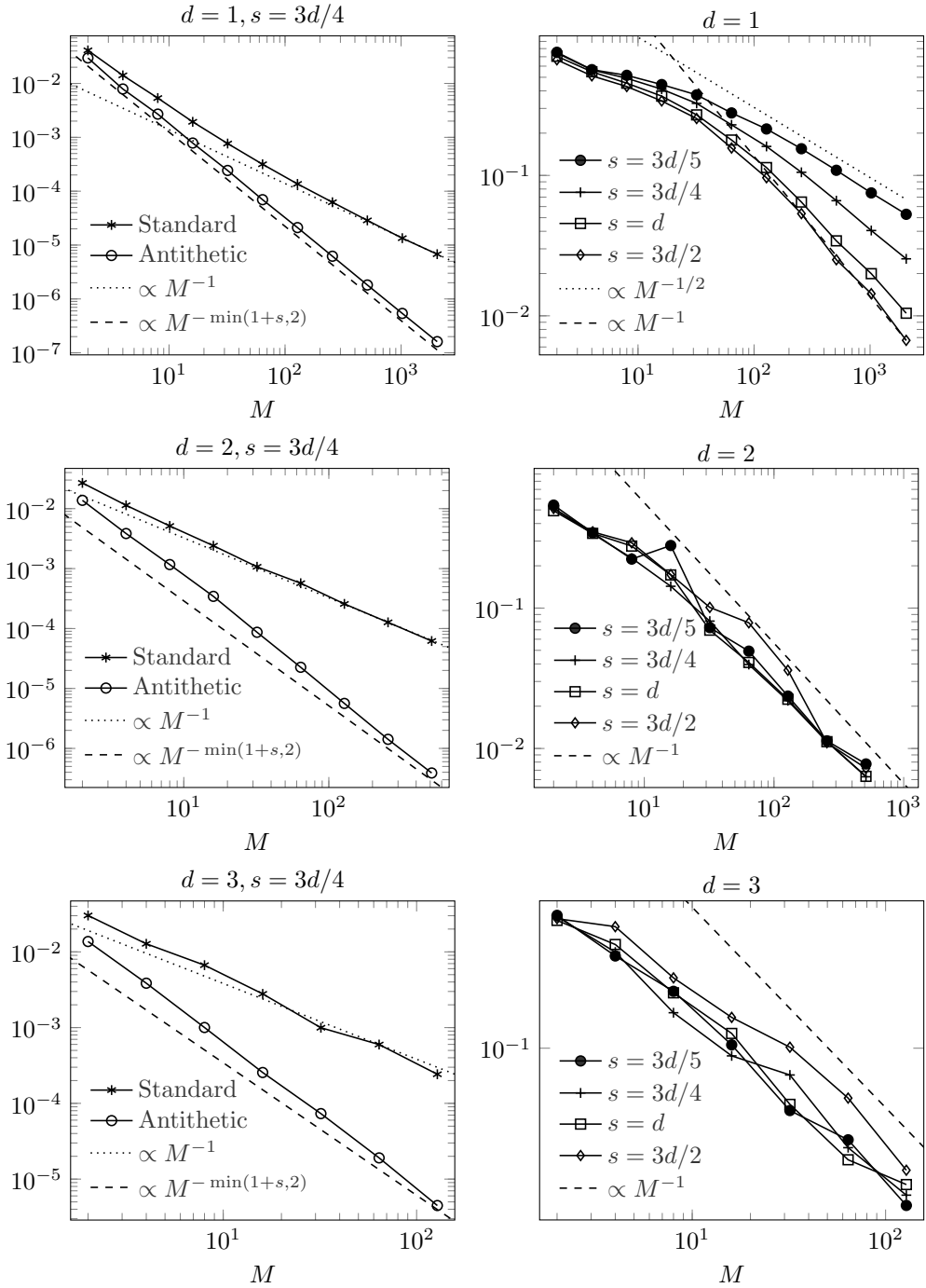
\begin{figure}
	\centering
\begin{tikzpicture}

  \begin{axis}[
    title={\(d=1, s=3d/4\)},
    xlabel={\(M\)},%
    xmin=1.41421356237309, xmax=2896.3093757401,
    ymin=9.15603897103308e-08, ymax=0.0765648656231737,
    xmode=log, ymode=log, log basis x={10}, log basis y={10}
    ]

    \addplot [semithick, black, mark=asterisk, mark size=2]
table {%
2 0.0411036853534655
4 0.01412021961891
8 0.00535725944131898
16 0.00193928687726122
32 0.000766421558446209
64 0.000316181571506849
128 0.000136567662484841
256 6.2354995182362e-05
512 2.86956111912606e-05
1024 1.34278383337412e-05
2048 6.77407221421778e-06
};
\addlegendentry{Standard}
\addplot [semithick, black, mark=o, mark size=2]
table {%
2 0.029762389708203
4 0.00788901430212326
8 0.00269600697153729
16 0.000783136771460127
32 0.000242504150239435
64 6.95947433872929e-05
128 2.1106789734055e-05
256 6.20688923318149e-06
512 1.8148179325381e-06
1024 5.38950471981205e-07
2048 1.6129836743655e-07
};
\addlegendentry{Antithetic}

\addplot [semithick, black, dotted, domain=1:4048]
{fitRefLine(x, -1, 2048, 6.77407221421778e-06)};
\addlegendentry{\(\propto M^{-1}\)}

\addplot [semithick, black, dashed, domain=1:4048]
{0.7 * fitRefLine(x, -1.75, 2048, 1.6129836743655e-07)};
\addlegendentry{\(\propto M^{-\min(1+s, 2)}\)}
\end{axis}

\end{tikzpicture}
\begin{tikzpicture}
  \begin{axis}[
    title={\(d=1\)},
    xlabel={\(M\)}, %
    xmin=1.41421356237309, xmax=2896.3093757401,
    ymin=0.0052704870167727, ymax=0.986947786557392,
    xmode=log, ymode=log, log basis x={10}, log basis y={10}
    ]

    \addplot [semithick, black, mark=*, mark size=2]
    table {%
      2 0.753377193928956
      4 0.5664723169933
      8 0.516124808689248
      16 0.443504120284092
      32 0.374597834309417
      64 0.278602714541589
      128 0.214190981152979
      256 0.154869071759841
      512 0.108765155343803
      1024 0.0750511796747396
      2048 0.0528491883201546
    };
    \addlegendentry{\(s=3d/5\)}

    \addplot [semithick, black, mark=+, mark size=2]
    table {%
      2 0.743399918371974
      4 0.564386575077325
      8 0.491848340196101
      16 0.411105595370899
      32 0.325782083912892
      64 0.229233870211102
      128 0.16085433957754
      256 0.105220331531698
      512 0.0660592237333161
      1024 0.0405079429742523
      2048 0.0254618692609877
    };
    \addlegendentry{\(s=3d/4\)}

    \addplot [semithick, black, mark=square, mark size=2]
    table {%
      2 0.705628188481
      4 0.547398487679
      8 0.45382735515
      16 0.367008160597
      32 0.269840331997
      64 0.178642777226
      128 0.113711391202
      256 0.064621441339
      512 0.034202782271
      1024 0.0199650736152
      2048 0.0104345812505
    };
    \addlegendentry{\(s=d\)}

    \addplot [semithick, black, mark=diamond, mark size=2]
    table {%
      2 0.665803423529132
      4 0.511723511942599
      8 0.429625847109029
      16 0.339110006341265
      32 0.254161316589963
      64 0.156645980305618
      128 0.0962561089944157
      256 0.0534406300032328
      512 0.0250767653726632
      1024 0.0143732261965499
      2048 0.00668564990556579
    };
    \addlegendentry{\(s=3d/2\)}

    \addplot [semithick, black, dotted, domain=2:2048]
    {0.9*fitRefLine(x, -0.5, 512, 0.149809252848076)};
    \addlegendentry{\(\propto M^{-1/2}\)}

    \addplot [semithick, black, dashed, domain=2:2048]
    {fitRefLine(x, -1, 2048, 0.00668564990556579)};
    \addlegendentry{\(\propto M^{-1}\)}

  \end{axis}
\end{tikzpicture}

\begin{tikzpicture}

  \begin{axis}[
    title={\(d=2, s=3d/4\)},
    xlabel={\(M\)},%
    xmin=1.5157165665104, xmax=675.588050315722,
    ymin=2.22831648204329e-07, ymax=0.046919487662473,
    xmode=log, ymode=log, log basis x={10}, log basis y={10}
    ]

    \addplot [semithick, black, mark=asterisk, mark size=2]
    table {%
      2 0.0268770504903491
      4 0.0113821809943314
      8 0.00514313482900699
      16 0.00240689048325706
      32 0.00107251061960379
      64 0.000569603377047274
      128 0.000257467582628775
      256 0.000126777989649216
      512 6.18606265915057e-05
    };
    \addlegendentry{Standard}
    \addplot [semithick, black, mark=o, mark size=2]
    table {%
      2 0.013762577282952
      4 0.00384599702389611
      8 0.00116684031551296
      16 0.000344046038848991
      32 8.65985975054433e-05
      64 2.2559852280592e-05
      128 5.63087470493069e-06
      256 1.41746635403084e-06
      512 3.88999037393844e-07
    };
    \addlegendentry{Antithetic}

    \addplot [semithick, black, dotted, domain=1:4048]
    {fitRefLine(x, -1, 256, 0.000126777989649216)};
    \addlegendentry{\(\propto M^{-1}\)}

    \addplot [semithick, black, dashed, domain=1:4048]
    {0.7 * fitRefLine(x, -1.75, 256, 1.41746635403084e-06)};
    \addlegendentry{\(\propto M^{-\min(1+s, 2)}\)}
  \end{axis}

\end{tikzpicture}
\begin{tikzpicture}
  \begin{axis}[
    title={\(d=2\)},
    xlabel={\(M\)}, %
    xmin=1.41421356237309, xmax=1296.3093757401,
    ymin=0.0052704870167727, ymax=0.986947786557392,
    xmode=log, ymode=log, log basis x={10}, log basis y={10}
    ]

    \addplot [semithick, black, mark=*, mark size=2]
    table {%
      2 0.541030683817432
      4 0.34378935313028
      8 0.223220856865484
      16 0.278980483893546
      32 0.0725925222612046
      64 0.0493225401453015
      128 0.0236147683751602
      256 0.0113313687236685
      512 0.00775814493255636
    };
    \addlegendentry{\(s=3d/5\)}

    \addplot [semithick, black, mark=+, mark size=2]
    table {%
      2 0.512056830339098
      4 0.337896315812542
      8 0.226873366984674
      16 0.142942124389233
      32 0.0807438135553706
      64 0.0396062474164714
      128 0.0218702278843759
      256 0.0111806975166024
      512 0.00628831388926246
    };
    \addlegendentry{\(s=3d/4\)}

    \addplot [semithick, black, mark=square, mark size=2]
    table {%
      2 0.493294120128674
      4 0.340467778114511
      8 0.276323486038796
      16 0.172167827654432
      32 0.0695857291757122
      64 0.0409861640747236
      128 0.0224231172419473
      256 0.0111760204886528
      512 0.00633334741196198
    };
    \addlegendentry{\(s=d\)}

    \addplot [semithick, black, mark=diamond, mark size=2]
    table {%
      2 0.512294989425928
      4 0.34894122223853
      8 0.291580965334704
      16 0.173097552223349
      32 0.1010791955103
      64 0.0785431777172327
      128 0.036075246930365
      256 0.0109870264780897
      512 0.0072253661074907
    };
    \addlegendentry{\(s=3d/2\)}

    \addplot [semithick, black, dashed, domain=1:4048]
    {2 * fitRefLine(x, -1, 256, 0.0109870264780897)};
    \addlegendentry{\(\propto M^{-1}\)}

  \end{axis}
\end{tikzpicture}

\begin{tikzpicture}

  \begin{axis}[
    title={\(d=3, s=3d/4\)},
    xlabel={\(M\)},%
    xmin=1.62450479271247, xmax=157.586484908149,
    ymin=2.88124841447487e-06, ymax=0.046979098707117,
    xmode=log, ymode=log, log basis x={10}, log basis y={10}
    ]

    \addplot [semithick, black, mark=asterisk, mark size=2]
    table {%
      2 0.0302298093981062
      4 0.0127756202225798
      8 0.00667291752782228
      16 0.00278863778801216
      32 0.000995818330791081
      64 0.000597571776549349
      128 0.000242256004658967
    };
    \addlegendentry{Standard}
    \addplot [semithick, black, mark=o, mark size=2]
    table {%
      2 0.0136930899389136
      4 0.00385697491795577
      8 0.00100316456481103
      16 0.000255372898671889
      32 7.32076386302427e-05
      64 1.90552981051308e-05
      128 4.47764826700559e-06
    };
    \addlegendentry{Antithetic}

    \addplot [semithick, black, dotted, domain=1:4048]
    {fitRefLine(x, -1, 64, 0.000597571776549349)};
    \addlegendentry{\(\propto M^{-1}\)}

    \addplot [semithick, black, dashed, domain=1:4048]
    {0.7 * fitRefLine(x, -1.75, 64, 1.90552981051308e-05)};
    \addlegendentry{\(\propto M^{-\min(1+s, 2)}\)}
  \end{axis}

\end{tikzpicture}
\begin{tikzpicture}
  \begin{axis}[
    title={\(d=3\)},
    xlabel={\(M\)}, %
    xmin=1.62450479271247, xmax=157.586484908149,
    ymin=0.0139035686062342, ymax=0.543700021520839,
    xmode=log, ymode=log, log basis x={10}, log basis y={10}
    ]

    \addplot [semithick, black, mark=*, mark size=2]
    table {%
      2 0.460240804241904
      4 0.288641521331773
      8 0.192270240553657
      16 0.103923050760024
      32 0.0488593209636708
      64 0.034906783268419
      128 0.0164248160544513
    };
    \addlegendentry{\(s=3d/5\)}

    \addplot [semithick, black, mark=+, mark size=2]
    table {%
      2 0.452966466264636
      4 0.301901187633842
      8 0.150333727433076
      16 0.0915762168072494
      32 0.0735150542690717
      64 0.031887881678691
      128 0.0184831260356537
    };
    \addlegendentry{\(s=3d/4\)}

    \addplot [semithick, black, mark=square, mark size=2]
    table {%
      2 0.434218124257592
      4 0.329272668417713
      8 0.188974311455352
      16 0.118493215826396
      32 0.05245253538134
      64 0.0277534652872989
      128 0.0208755408236445
    };
    \addlegendentry{\(s=d\)}

    \addplot [semithick, black, mark=diamond, mark size=2]
    table {%
      2 0.442915066923684
      4 0.404238995892589
      8 0.224399006616765
      16 0.142481546366403
      32 0.101150303579019
      64 0.0563720781739241
      128 0.0246073728258529
    };
    \addlegendentry{\(s=3d/2\)}

    \addplot [semithick, black, dashed, domain=1:4048]
    {1.4 * fitRefLine(x, -1, 64, 0.0563720781739241)};
    \addlegendentry{\(\propto M^{-1}\)}

  \end{axis}
\end{tikzpicture}
 
	\caption{ Results for numerical example in Section~\ref{sec:numerics}
          and \(M,N \) and \(K\) as in \eqref{eq:balance}. \textit{(left)}
          Shows the left-hand sides of \eqref{eq:var-decay}, the variance for
          the antithetic estimator, and \eqref{eq:var-decay-EM}, the variance
          for the ``Standard'' truncated Milstein estimator without the
          antithetic correction, for the smoothness parameter \(s=3d/4\).
          \textit{(right)} Shows the relative variance decay between the two
          estimators, i.e., \(\max_{m}\E{\norm{\ol Y_m-Y_m^c}_H^2} \big /
          \max_{m}\E{\norm{Y^{f}_m-Y_m^c}_H^2} =\mathcal O(M^{-\min(s, 1)})
          \), for different smoothness parameters \(s\). The variance
          estimates were obtained using Monte Carlo sampling with at least
          4000 samples.}
	\label{fig:variance-M-plots}
\end{figure}

\section{Conclusions}
\label{sec:conlcusions}

We have developed an antithetic MLMC-Milstein scheme for parabolic SPDEs, which offers a significant improvement in computational efficiency for estimating quantities of interest in SPDE models. This scheme circumvents the need to simulate intractable Lévy area terms, making it particularly advantageous for SPDEs with multiplicative noise and non-commutative diffusion terms.
In our study, we have derived precise variance decay bounds for a fully discrete scheme that incorporates antithetic time stepping, spatial approximations, and noise approximations. Furthermore, we have bounded the computational effort by considering the cost associated with evaluating the noise term. These results provide valuable insights into the efficiency and accuracy of our proposed scheme.

There are several possible extensions to the current work that could be explored. A further step to enhance efficiency would be to develop a higher-order noise approximation that achieves a better rate than $\mathcal{O}(K^{-\gamma})$
in relation to the truncation index $K$ (cf. Theorem~\ref{thm:MainRes}).
Additionally, the methodology employed in this study could be extended to incorporate discontinuous Lévy driving noise, provided that it possesses a sufficient number of moments. While this extension may initially seem straightforward, it is important to emphasize that our results heavily rely on the continuous version of the Burkholder-Davis-Gundy inequality (Eq.~\ref{eq:BDG2}), while only a weaker version (Eq.~\ref{eq:BDG1}) is available for discontinuous martingales. Consequently, a completely different proof technique would be required, even for the relatively simple case of Poisson driving noise.

Another intriguing avenue for exploration would be the consideration of first-order hyperbolic SPDEs, which commonly arise in the modeling of energy forward contracts \cite{BK08, BB14}. In such cases, the weak formulation of SPDEs becomes essential for pathwise discretizations, see \cite{BS2019stochastic}, as the associated semigroups lack the smoothing properties observed in the parabolic case.
Furthermore, recent developments have seen the application of a modified
version of the antithetic Milstein scheme to finite-dimensional stochastic
differential equations with non-Lipschitz drift \cite{pang2023antithetic}.
Extending this result to SPDEs in infinite dimensions would be both intriguing
and worthwhile. \add{Finally, an enhanced Milstein scheme which does not
  require the evaluation of derivatives of \(G\) has been proposed in
  \cite{rossler:deriviative-free-milstein-SPDE-commutative}, and in
  \cite{rossler:deriviative-free-milstein-SPDE-non-commutative} the case of
  non-commutative noise is addressed by approximating the iterated integrals.
  Constructing a truncated version of these enhanced schemes coupled with an
  antithetic estimator would be a logical next step to reduce the cost of the
  scheme and improve computational complexity of the MLMC estimator in
  problems where evaluating derivatives of \(G\) is costly.}

\subsection*{Acknowledgements}

Work on this manuscript was initiated during the authors' visit to the
workshop ``Computational Uncertainty Quantification: Mathematical Foundations,
Methodology \& Data'' at the Erwin-Schrödinger Institute (ESI) in Vienna. We
would like to thank the organizers of this workshop and the ESI for their
invitation and their hospitality during this time.
AS is partly funded by ETH Foundations of Data Science (ETH-FDS), and it is gratefully acknowledged.

\appendix
\section{It\^o Isometry and Burkholder-Davis-Gundy Inequalities}
\label{sec:isometries}

We record the following It\^o isometry  and Burkholder-Davis-Gundy (BDG)-type inequalities for Hilbert space-valued stochastic integrals in the setting of Section~\ref{sec:prelim}.

\begin{lem}\label{lem:ito-BDG}
  Let $Y\in \cM^2(U)$ \add{and denote its martingale covariance by \(\cQ_{Y}\)}. Let
  $G:\gO\times\bT\to \cL(U, H)$ be a $\cP_\bT/\cB(\LHS(\cU, H))$-measurable process
  such that
	\begin{equation*}
		\E{\int_0^T \norm{G(s)\cQ_Y^{\nicefrac{1}{2}}(s)}_{\LHS(U, H)}^2\,d\bracket{Y}{Y}_s}<\infty.
	\end{equation*}
	\begin{enumerate}
		\item There holds the isometric formula
		\be\label{eq:ito-isometry}
		\E{\left\|\int_0^t  G(s)dY(s) \right\|_H^2}
		=\E{\int_0^T \norm{G(s)\cQ_Y^{\nicefrac{1}{2}}(s)}_{\LHS(U, H)}^2\,d\bracket{Y}{Y}_s}, \quad t\in\bT.
		\ee

		\item  If for some $p>2$ there holds
		\begin{equation*}
			\E{\int_0^T \norm{G(s)\cQ_Y^{\nicefrac{1}{2}}(s)}_{\LHS(U, H)}^p\,d\bracket{Y}{Y}_s}<\infty,
		\end{equation*}
		then there is a $C=C(T, p)>0$ such that
		\be\label{eq:BDG1}
		\E{\left\|\int_0^t  G(s)dY(s) \right\|_H^p}
		\le C \E{\int_0^T \norm{G(s)\cQ_Y^{\nicefrac{1}{2}}(s)}_{\LHS(U, H)}^p\,d\bracket{Y}{Y}_s}, \quad t\in\bT.
		\ee
		Moreover, if $Y$ has continuous trajectories, then
		\be\label{eq:BDG2}
		\E{\left\|\int_0^t  G(s)dY(s) \right\|_H^p}
		\le C \E{\LR{\int_0^T \norm{G(s)\cQ_Y^{\nicefrac{1}{2}}(s)}_{\LHS(U, H)}^2\,d\bracket{Y}{Y}_s}^{p/2}}, \quad t\in\bT.
		\ee
	\end{enumerate}
\end{lem}

For a proof of~\eqref{eq:ito-isometry} see e.g. \cite[Theorem 8.7]{PZ07}, the BDG inequality~\eqref{eq:BDG2} for continuous martingales may for instance be found \cite[Eq. (1.5)]{hausenblas2007stochastic} and the references therein.
The previous result simplifies for Wiener processes with constant martingale covariance $\cQ_Y=Q\, \tr(Q)^{-1}$ as in Section~\ref{sec:Wiener}.
In this case, the It\^o isometry and BDG inequality from Lemma~\ref{lem:ito-BDG} admit the following form.

\begin{cor}\cite[Corollary 8.17 and Lemma 8.27]{PZ07}\label{cor:ito-BDG}
	Let $W$ be a $Q$-Wiener process and let $G:\gO\times\bT\to \LHS(\cU, H)$ be a
	$\cP_\bT/\cB(\LHS(\cU, H))$-measurable and square-integable mapping.
	\begin{enumerate}
		\item There holds the isometric formula
		\be\label{eq:ito-isometry-W}
		\bE\left(\left\|\int_0^t  G(s)dW(s) \right\|_H^2\right)
		=\int_0^t\bE\left(\|G(s)\|^2_{\LHS(\cU,H)}\right)ds
		=\int_0^t\sum_{k\in\bN}\eta_k\bE\left(\|G(s)e_k\|^2_H\right)ds.
		\ee

		\item If, in addition, for some $p>2$ there holds
		\bee
		\bE\left(\int_0^t\|G(s)\|^p_{\LHS(\cU,H)}ds\right)<\infty,
		\eee
		then there is a $C=C(p)>0$ such that for $t\in\bT$ there holds
		\be\label{eq:BDG-W}
		\bE\left(\left\|\int_0^t  G(s)dW(s) \right\|_H^p\right)
		\le C \bE\LR{\left(\int_0^t\|G(s)\|^2_{\LHS(\cU,H)}ds\right)^{p/2}}.
		\ee
	\end{enumerate}
\end{cor}

\section{Proofs of Section~\ref{sec:approximation}}
\label{sec:app1}

\begin{proof}[Proof of Proposition~\ref{prop:correction_est}]
	For any fixed $X\in H$, define
	\bee
	\widetilde G_X:U\times U \del{\mapsto}\add{\to} H,\quad  (\phi, \varphi)\mapsto \frac{1}{2}G'(X)(P_NG(X)\phi)\varphi.
	\eee
	As $G'(X)\in \cL(H, \cL(\cU, H))$ and $G(X)\in \cL(\cU, H)$, it readily follows that $\widetilde G_X$ is bilinear.
	Thus, there exists a unique $\cG_X\in\cL(\cU\otimes\cU, H)\simeq\cL(\LHS(\cU), H)$  such that
	\be\label{eq:tensor_G}
	\widetilde G_X(\phi, \varphi)=\cG_X(\phi\otimes \varphi),\quad (\phi, \varphi)\in U\times U.
	\ee
	We thus define $\cG:H\to \cL(\LHS(\cU), H),\; X\mapsto \cG_X$, and~\eqref{eq:tensor_correction} follows by the linearity of $\cG_X$ together with
	\begin{align*}
		\gD_m\cW_{m,K}
		&=(W_K(t_{m+1})-W_K(t_m))\otimes (W_K(t_{m+1})-W_K(t_m)) - \gD t\sum_{k=1}^K\eta_k\, e_k\otimes e_k \\
		&=\sum_{k,l=1}^K
		(\gD_m w_k\gD_m w_l-\gd_{k,l}\eta_k\gD t)\,e_k\otimes e_l.
	\end{align*}

	To show that $\cG(X)\in \LHS(\LHS(\cU), H)$ for all $X\in H$, we use~\eqref{eq:tensor_G} and that $\LR{\sqrt{\eta_k}e_k\otimes \sqrt{\eta_l}e_l, (k,l)\in\bN^2}$ is an orthonormal basis of $\LHS(\cU)$ to obtain
	\begin{align*}
		\|\cG(X)\|_{\LHS(\LHS(\cU), H)}^2
		&=
		\sum_{k,l\in\bN} \|\cG(X)(\sqrt{\eta_k}e_k\otimes \sqrt{\eta_l}e_l)\|_H^2 \\
		&=
		\frac{1}{4}\sum_{k,l\in\bN} \|G'(X)(P_NG(X)\sqrt{\eta_k}e_k)\sqrt{\eta_l}e_l\|_H^2 \\
		&=
		\frac{1}{4}\sum_{k\in\bN} \|G'(X)(P_NG(X)\sqrt{\eta_k}e_k)\|_{\LHS(\cU, H)}^2 \\
		&\le
		\frac{1}{4}\|G'(X)\|_{\cL(H, \LHS(\cU, H))}^2\sum_{k\in\bN} \|G(X)\sqrt{\eta_k}e_k\|_H^2\\
		&=
		\frac{1}{4}\|G'(X)\|_{\cL(H, \LHS(\cU, H))}^2\|G(X)\|_{\LHS(\cU, H)}^2.
	\end{align*}
	Using that $G$ is twice differentiable with bounded derivatives from
	Assumption~\ref{ass:SPDE}~\ref{item:FG-diff}, we obtain
	\begin{align*}
		\|\cG(X)\|_{\LHS(\LHS(\cU), H)}^2 \le C(1+\|X\|_H)^2.
	\end{align*}
	The bound in~\eqref{eq:tensor_norm2} follows analogously, by using
	Assumption~\ref{ass:SPDE}~\ref{item:smooth-coefficients} instead.

	The Fr\'echet derivative $\cG'\in \cL(H\times H, \LHS(\LHS(\cU), H))$ of $\cG$ is for all $(X,Y)\in H\times H$ given by
	\begin{equation*}
		\cG'(X)(Y)(\phi\otimes\varphi) =
		\frac{1}{2}\LR{G''(X)[P_NG(X)\phi, Y]\varphi+G'(X)(P_NG'(X)(Y)\phi)\varphi},\quad \phi,\varphi\in\cU.
	\end{equation*}
	The last estimate then follows since $G'$ and $G''$ are globally bounded by Assumption~\ref{ass:SPDE}~\eqref{item:FG-diff}, since
	\begin{align*}
		&\|\cG'(X)(Y)(\phi\otimes\varphi)\|_{\LHS(\LHS(\cU), H)} \\
		&\le
		\frac{1}{2}\Big(\|G''(X)[P_NG(X)\phi, Y]\|_{\LHS(\cU, H)}
		+\|G'(X)(P_NG'(X)(Y)\phi)\|_{\LHS(\cU, H)}\Big)\|\varphi\|_\cU \\
		&\le
		\frac{1}{2}\Big(\|G''(X)\|_{\cL(H\times H, \LHS(\cU, H))}
		\|G(X)\phi\|_H\|Y\|_H +\|G'(X)\|_{\cL(H, \LHS(\cU, H))}
		\|G'(X)(Y)\phi\|_H\Big)\|\varphi\|_\cU \\
		&\le
		\frac{1}{2}\Big(\|G''(X)\|_{\cL(H\times H, \LHS(\cU, H))}
		\|G(X)\|_{\LHS(\cU, H))}\|Y\|_H \\
		&\quad\qquad+\|G'(X)\|_{\cL(H, \LHS(\cU, H))}
		\|G'(X)\|_{\cL(H, \LHS(\cU, H))}\|Y\|_H \Big)\|\phi\|_\cU\|\varphi\|_\cU \\
		&\le C(1+\|X\|_H)\|Y\|_H \|\phi\otimes\varphi\|_{\cL_1(\cU)}.
	\end{align*}
\end{proof}

We next record some stability and error estimates on the rational approximation $r(\gD t A_N)\approx S_N(\gD t)$ to prove Theorem~\ref{thm:stability}.

\begin{lem}\label{lem:rational-approx}
  Let Assumption~\ref{ass:approximation}~\ref{item:rational}
  and~\ref{item:subspace_approx} hold.
  \begin{enumerate}
  \item For any $N\in\bN$, $\gD t>0$, \add{$j\in\bN$} and $\ga\ge 0$ there holds
		\begin{equation*}
			\|r(\gD t A_N)^{\add{j}}P_N\|_{\cL(H)}\le 1 \quad\text{and}\quad
			\|r(\gD t A)^{\add{j}}\|_{\cL(\dot H^\ga)}\le 1.
		\end{equation*}

              \item For any $\ga\in [0,\del{2(q+1)}\add{4}]$ there exists $C>0$ such that for
                any $\gD t>0$ there holds
		\begin{equation*}
			\|r(\gD tA)-I\|_{\cL(\dot H^\ga, H)} \le C \gD t^{\nicefrac{\ga}{2}}.
		\end{equation*}

		\item For any $\ga\in [0,\del{2(q+1)}\add{4}]$ there exists $C$ such that for any $N, \gD t>0$, $j\in\bN$ and $v\in \dot H^\ga$ it holds
		\begin{equation*}
			\|(r(\gD t A_N)^j-S_N(j\gD t))P_N v\|_H\le
			C \gD t^{\nicefrac{\min(\ga, 2)}{2}}\|v\|_{\dot H^{\min(\ga, 2)}}.
		\end{equation*}
		For $\mfd t:=\frac{\gD t}{2}$ and $j\in\bN$ there holds
		\begin{equation*}
			\|(r(\mfd tA_N)^{2j}-r(\gD t A_N))^jP_N v\|_H\le C \gD t^{\nicefrac{\min(\ga, 2)}{2}}\|v\|_{\dot H^{\min(\ga, 2)}}.
		\end{equation*}
		Furthermore, there exists $\widetilde \ga>0$ such that
		\begin{equation*}
			\|(r(\gD t A_N)P_N-I)v\|_H\le C (\gD t^{\nicefrac{\min(\ga, 2)}{2}} + N^{-\widetilde \ga})\|v\|_{\dot H^\ga}.
		\end{equation*}
	\end{enumerate}
\end{lem}

\begin{proof}
	~
	\begin{enumerate}
		\item These stability estimates are well-known and may be found for instance in the proof of~\cite[Theorem 7.1]{thomee2007galerkin}. We give a short proof here for the reader's convenience.

		Let $(\widetilde f_1, \dots, \widetilde f_n)$ denote the eigenbasis of $(-A_N)$, and recall that the corresponding eigenvalues satisfy $\widetilde \gl_n>0$ for all $n=1,\dots,N$. Since $r(z)<1$ for all $z\ge 0$ by the first part of Assumption~\ref{ass:approximation}~\ref{item:rational}, we have for all $v\in H$ that
		\begin{equation*}
			\|r(\gD t A_N)^{\add{j}}P_N v\|_H^2
			= \left\| \sum_{n=1}^{N} r(\gD t \widetilde\gl_n)^{\add{j}} (P_Nv, \widetilde f_n)_H \widetilde f_n\right\|_H^2
			\le \|P_N v\|_H^2 \le \|v\|_H^2.
		\end{equation*}
		For the second part, let $\gl_n>0$ and $f_n\in H$ denote for $n\in\bN$ the eigenvalues and eigenfunctions of $(-A)$.
		Similar as for the first part, we have
		\begin{equation*}
			\|r(\gD t A) ^{\add{j}}v\|_{\dot H^\ga}^2
			=  \sum_{n\in\bN} \gl_n^{\ga}|r(\gD t \gl_n)^{\add{j}}|^2 |(v, f_n)_H|^2
			\le \sum_{n\in\bN} \gl_n^{\ga}|(v, f_n)_H|^2
			\le \|v\|_{\dot H^\ga}^2.
		\end{equation*}

		\item The triangle inequality yields for any $v\in \dot H^\ga$ that
		\begin{equation*}
			\|r(\gD tA)-I\|_{\cL(\dot H^\ga, H)} \le \|r(\gD tA)-S(\gD t)\|_{\cL(\dot H^\ga, H)}
			+ \|S(\gD t)-I\|_{\cL(\dot H^\ga, H)} \le C \gD t^{\nicefrac{\ga}{2}}.
		\end{equation*}
		The first term on the right hand side is bounded by \cite[Theorem 7.1]{thomee2007galerkin}, the second term by \cite[Theorem 6.13, part d)]{pazy1983semigroups}.

              \item The first part is again given in \cite[Theorem 7.1]{thomee2007galerkin} together with Assumption~\ref{ass:approximation}~\ref{item:subspace_approx}\add{.}\rdel{ that yields $$\norm{A_N^{\nicefrac{\min(\ga, 2)}{2}}P_N v}_H \le C \norm{v}_{\dot{H}^{\min(\ga, 2)}}.$$}
		The second part then follows immediately by the triangle
                inequality, since $S_N(2j\mfd t)= S_N(j\gD t)$. The final
                estimate follows by
		\begin{align*}
			\|(r(\gD t A_N)P_N-I)v\|_H
			&\le
			\norm{(r(\gD t A_N)-I)P_Nv}_H + \norm{(P_N-I)v}_H \\
			&\le
			C \gD t^{\nicefrac{\min(\ga, 2)}{2}}\norm{A_N^{\nicefrac{\min(\ga, 2)}{2}}P_N v}_H + \norm{(P_N-I)v}_H\\
			&\le
			C \LR{\gD t^{\nicefrac{\min(\ga, 2)}{2}} + N^{-\widetilde \ga}}\norm{v}_{\dot{H}^\ga},
		\end{align*}
		where the second estimate is derived in the same fashion as part 2.),
		and we have used
		Assumption~\ref{ass:approximation}~\ref{item:subspace_approx} in the last step.
	\end{enumerate}
\end{proof}

\begin{proof}[Proof of Theorem~\ref{thm:stability}]
	1. For $m=0,\dots, M-1$, we re-iterate the representation in~\eqref{eq:truncated_milstein} to obtain
	\be\label{eq:milstein_iterated}
	\begin{split}
		Y_{m+1}^{N,K}
		&= r(\gD tA_N)^m P_NX_0 +
		\sum_{j=0}^m r(\gD tA_{N_f})^jP_N F(Y_{m-j}^{N,K}) \gD t
		+\sum_{j=0}^m r(\gD tA_{N_f})^jP_N G(Y_{m-j}^{N,K})\gD_{m-j}W_K \\
		&\quad+ \sum_{j=0}^m r(\gD tA_N)P_N \cG(Y_{m-j}^{N,K})\gD_{m-j} \cW_{m-j, K} \\
		&=: r(\gD tA_N)^m P_NX_0 + \tI + \tII + \tIII.
	\end{split}
	\ee
	The first term is bounded in $L^p(\gO; H)$ by Jensen's inequality and the first part of Lemma~\ref{lem:rational-approx}
	\begin{align*}
		\E{\norm{\tI}_H^p}
		\le \gD t^{p} m^{p-1} \sum_{j=0}^m \E{\norm{r(\gD tA_{N_f})^jP_N F(Y_{m-j}^{N,K}) }_H^p}
		\le C \gD t \sum_{j=0}^m \E{1+\norm{Y_{m-j}^{N,K}}_H^p},
	\end{align*}
	where the last bound holds since $F$ is of linear growth.

	For the second term, we use the BDG inequality in~\eqref{eq:BDG-W} together with Jensen's inequality and the linear growth of $G$ to obtain
	\begin{align*}
		\E{\norm{\tII}_H^p}
		&\le C \E{ \LR{\sum_{j=0}^m \gD t \norm{r(\gD tA_{N_f})^jP_N G(Y_{m-j}^{N,K}) }_{\LHS(\cU, H)}^2}^{p/2} } \\
		&\le C \gD t^{p/2}m^{p/2-1} \sum_{j=0}^m \E{\norm{r(\gD tA_{N_f})^jP_N G(Y_{m-j}^{N,K}) }_{\LHS(\cU, H)}^p} \\
		&\le C \gD t \sum_{j=0}^m \E{1+\norm{Y_{m-j}^{N,K}}_H^p}.
	\end{align*}

	To bound the last term, recall the bound in~\eqref{eq:W_bracket}, which shows with the BDG inequality from~\eqref{eq:BDG2} that
	\begin{align*}
		&\E{\norm{\tIII}_H^p} \\
		&\le C \E{ \LR{
				\sum_{j=0}^m \int_{t_j}^{t_{j+1}}
                  \norm{r(\gD tA_{N_f})^jP_N \cG(Y_{m-j}^{N,K}) Q_{\cW_{m-j, K}}^{\nicefrac{1}{2}}(s)}_{\LHS(
                  \del{U\otimes U}\add{\LHS(U)}, H)}^2 \,d\bracket{\cW_{m-j, K}}{\cW_{m-j, K}}_s}^{p/2} }   \\
		&\le C \E{ \LR{ \sum_{j=0}^m \norm{\cG(Y_{m-j}^{N,K})}_{\LHS(\del{\cU\otimes\cU}\add{\LHS(\cU)}, H)}^2 \gD t^2 }^{p/2} }   \\
		&\le C \gD t^{p}m^{p/2-1} \sum_{j=0}^m \E{\norm{\cG(Y_{m-j}^{N,K})}_{\LHS(\del{\cL_1(\cU)}\add{\LHS(\cU)}, H)}^p}  \\
		&\le C \gD t^{p/2+1} \sum_{j=0}^m \E{1+\norm{Y_{m-j}^{N,K}}_H^p} ,
	\end{align*}
	where the last bound holds due to the linear growth
        bound~\eqref{eq:tensor_norm}.

	\del{Substituting the estimate for $\tI{-}\tIII$ into~\eqref{eq:milstein_iterated} and taking expectations yields}
        \add{Taking
          expectations
          of \eqref{eq:milstein_iterated} and
          substituting the estimates for $\tI{-}\tIII$ into the result yields}
	\bee
	\E{\norm{Y_{m+1}^{N,K}}_H^p}\le C(1 + \E{\norm{X_0}_H^p} + \gD t\sum_{j=0}^m\E{\norm{Y_{m-j}^{N,K}}_{H}^p}),
	\eee
	and the first part of Theorem~\ref{thm:stability} follows by the discrete Grönwall inequality.

	2. Let $e_m:= Y_m^{N,K}-\widehat Y_m^{N,K}$ for $m=0,\dots, M$ to obtain the error representation
	\bee
	\begin{split}
		\widehat e_{m+1} &=
		\sum_{j=0}^m r(\gD tA_{N_f})^jP_N\left(F(Y_{m-j}^{N,K}) - F(\widehat Y_{m-j}^{N,K})\right)\gD t \\
		&\quad+\sum_{j=0}^m r(\gD tA_{N_f})^jP_N\left(G(Y_{m-j}^{N,K}) - G(\widehat Y_{m-j}^{N,K})\right)\gD_{m-j}W_{M} \\
		&\quad+ \sum_{j=0}^m r(\gD tA_N)P_N \cG(Y_{m-j}^{N,K})\add{\gD_{m-j}}\cW_{m-j, K}
	\end{split}
	\eee
	Using that $F$ and $G$ are Lipschitz and repeating arguments from the first part of the proof yields
	\bee
	\begin{split}
		\E{\norm{e_{m+1}}_H^p}\le
		C\LR{\gD t\sum_{j=0}^m\E{\norm{e_j}_{H}^p} + \gD t^{p/2+1} \sum_{j=0}^m \E{\add{1+}\norm{Y_j^{\add{N,K}}}_{H}^p}}.
	\end{split}
	\eee
	The claim then follows with Grönwall's inequality, since we have shown that $\E{\norm{Y_j^{\add{N,K}}}_{H}^p}<\infty$ is uniformly bounded with respect to $\gD t$ in the first part.
\end{proof}

\section{Proof of Theorem~\ref{thm:MainRes} -- Antithetic Variance Decay}
\label{sec:app2}

Our strategy to prove Theorem~\ref{thm:MainRes} closely follows the approach in \cite{giles2014antithetic}. We bound $\bE(\|\ol Y_m - Y_m^c\|_H^2)$ by deriving appropriate difference equations of the antithetic average in~\eqref{eq:antiavg}
and by bounding higher-order remainder terms.
We introduce the \emph{semi-discrete temporal} fine discretizations $\widetilde Y^f:\gO\times\{0,\nicefrac{1}{2}, 1, \dots,M-\nicefrac{1}{2},M\}\to H$ via $\widetilde Y_0^f=X_0$,
\be\label{eq:semi-fine1}
\begin{split}
	\widetilde Y_{m+1/2}^f
	&= r(\mfd tA)\LR{ \widetilde Y_m^f
		+ F(\widetilde Y_m^f)\mfd t+G(\widetilde Y_m^f)\mfd_m W_{K_f}
		+ \cG(\widetilde Y_m^f)\mfd_m\cW_{m,K_f} },
\end{split}
\ee
and
\be\label{eq:semi-fine2}
\begin{split}
	\widetilde Y_{m+1}^f
	&= r(\mfd tA) \LR{ \widetilde Y_{m+1/2}^f
		+ F(\widetilde Y_{m+1/2} ^f)\mfd t+G(\widetilde Y_{m+1/2} ^f)\mfd_{m+1/2}
		+ \cG(\widetilde Y_{m+1/2} ^f)\mfd_{m+1/2} \cW_{m,K_f} }.
\end{split}
\ee
Note that we have used $A$ instead of $A_{N_f}$ as compared to  Equations~\eqref{eq:semi-fine1} and~\eqref{eq:semi-fine2}, hence $\widetilde Y_\cdot^f$ involves temporal and noise discretization, but no spatial approximation.
The corresponding antithetic fine semi-discretizations $\widetilde Y_\cdot^a$ are defined analogously by replacing $A_{N_f}$ by $A$ and $P_{N_f}$ by $I$ in~\eqref{eq:anti1} and~\eqref{eq:anti2}.

The next two auxiliary results establish a bound on $\widetilde Y_\cdot^f-Y_\cdot^f$;
\begin{lem}\label{lem:diff}
	Let Assumptions~\ref{ass:SPDE} and~\ref{ass:approximation}~\ref{item:rational} hold.
	For any $p\in(0,8]$ there is a constant $C=C(p)>0$ such that for all  $M, N_f, K_f\in\bN$ and $m=0,\nicefrac{1}{2}, 1, \dots,M-\nicefrac{1}{2}$ there holds
	\begin{equation}
		\E{\norm{Y_{m+1/2}^f-r(\mfd tA_{N_f}) Y_m^f}_H^p}
		+
		\E{\norm{Y_{m+1/2}^a-r(\mfd tA_{N_f}) Y_m^a}_H^p}\le C M^{-\nicefrac{p}{2}}.
	\end{equation}
\end{lem}

\begin{proof}
	We have by Equation~\eqref{eq:fine1}, Lemma~\ref{lem:rational-approx} and Jensen's inequality that
	\begin{align*}
		&\E{\norm{Y_{m+1/2}^f-r(\mfd tA_{N_f}) Y_m^f}_H^p}\\
		&\le
		C\LR{\E{\norm{r(\mfd tA_{N_f})P_{N_f} F(Y_m^f)\mfd t}_H^p}
			+ \E{\norm{r(\mfd tA_{N_f})P_{N_f} G(Y_m^f)\mfd_m W_{K_f}}_H^p}} \\
		&\quad+
		C\E{\norm{\frac{r(\mfd tA_{N_f})P_N}{2}\cG(Y_m^f)\mfd_m\cW_{m,K_f}}_H^p}\\
		&\le C\LR{\E{\norm{F(Y_m^f)}_H^p} \mfd t^p
			+\E{\norm{G(Y_m^f)\mfd_m W_{K_f}}_H^p}}+C\E{\norm{\cG(Y_m^f)\mfd_m\cW_{m,K_f}}_H^p}.
	\end{align*}

	The second part of Corollary~\ref{cor:ito-BDG} shows together with~\eqref{eq:W_bracket}, Assumption~\ref{ass:SPDE}~\ref{item:FG-diff} and Proposition~\ref{prop:correction_est} that
	\begin{align*}
		\E{\norm{Y_{m+1/2}^f-r(\mfd tA_{N_f}) Y_m^f}_H^p}
		&=
		C\LR{\mfd t^p \E{1+\norm{Y_m^f}_H^p}
			+ \mfd t^{p/2} \E{\norm{G(Y_m^f)}_{\LHS(\cU,H)}^p} } \\
		&\quad+C \mfd t^p \E{\norm{\cG(Y_m^f)}_{\LHS(\LHS(\cU),H)}^p} \\
		&\le C \mfd t^{p/2} \E{1+\norm{Y_m^f}_H^p}  \\
		&\le C \mfd t^{p/2}.
	\end{align*}
	For the last step we have used that $\E{\norm{Y_m^f}_H^p}$ is uniformly bounded by Theorem~\ref{thm:stability}. The bound for $\E{\norm{Y_{m+1/2}^a-r(\mfd tA_{N_f}) Y_m^a}_H^p}$ follows analogously.
\end{proof}

\begin{lem}\label{lem:semi-error}
	Let Assumptions~\ref{ass:SPDE} and~\ref{ass:approximation} hold.
	Then, there is a constant $C=C(p)>0$ such that for all  $M, N_f, K_f\in\bN$ and $m=0,\nicefrac{1}{2}, 1, \dots,M-\nicefrac{1}{2}, M$ there holds
	\begin{alignat*}{2}
		\E{\norm{\widetilde Y_m^f}_{\dot H^\ga}^p+\norm{\widetilde Y_m^a}_{\dot H^\ga}^p}
		&\le C,\quad&&\text{for $p\in(0,8]$, \quad and}\\
		\E{\norm{Y_m^f-\widetilde Y_m^f}_H^p+\norm{Y_m^a-\widetilde Y_m^a}_H^p}
		&\le C \LR{M^{-p} + N_f^{-p\widetilde \ga} + K_f^{-p\del{2}\gb}}\quad&&\text{for $p\in(0,4]$}.
	\end{alignat*}
\end{lem}

\begin{proof}
	We may represent $\widetilde Y_m^f$ for any $m=0,\nicefrac{1}{2}, 1, \dots,M-\nicefrac{1}{2}, M$ by the expansion
	\be\label{eq:semi-fine3}
	\begin{split}
		\widetilde Y_m^f
		&= r(\mfd tA)^{2m} \widetilde Y_0^f
		+ \sum_{j=0}^{2m-1}r(\mfd tA)^{2n-j}\LR{
			F(\widetilde Y_{j/2}^f)\mfd t
			+G(\widetilde Y_{j/2}^f)\mfd_{j/2} W_{K_f}
			+\cG(\widetilde Y_{j/2}^f)\mfd_{j/2} \cW_{\floor{j/2},K_f} }.
	\end{split}
	\ee
	This in turn shows with the first part of Lemma~\ref{lem:rational-approx}, Jensen's inequality, and the second part of Corollary~\ref{cor:ito-BDG} for $p\in[2,8]$ that
	\begin{align*}
		\E{\norm{\widetilde Y_m^f}_{\dot H^\ga}^p}
		&\le
		C\left(
		\E{\norm{\widetilde Y_0^f}_{\dot H^\ga}^p}
		+
		\sum_{j=0}^{2m-1}
		\E{\norm{F(\widetilde Y_{j/2}^f)}_{\dot H^\ga}^p} (2m)^{p-1} \mfd t^p \right) \\
		&\quad +
		C \sum_{j=0}^{2m-1} \E{\norm{G(\widetilde Y_{j/2}^f)}_{\LHS(\cU, \dot H^\ga)}^p} (2m)^{p/2-1} \mfd t^{p/2}    \\
		&\quad + C\sum_{j=0}^{2m-1}
		\E{\norm{\cG(\widetilde Y_{j/2}^f)}_{\LHS(\cL_1(\cU), \dot H^\ga)}^p}
		(2m)^{p/2-1} \mfd t^{p}\\
		&\le C\left(
		\E{\norm{\widetilde Y_0^f}_{\dot H^\ga}^p}
		+
		\mfd t \sum_{j=0}^{2m-1}
		\E{1+\norm{\widetilde Y_{j/2}^f}_{\dot H^\ga}^p} \right) \\
		&\le C\left(
		\E{\norm{X_0}_{\dot H^\ga}^p}
		+ 1+ \mfd t \sum_{j=0}^{2m-1}
		\E{\norm{\widetilde Y_{j/2}^f}_{\dot H^\ga}^p}\right).
	\end{align*}
	We have used Assumption~\ref{ass:SPDE}~\ref{item:smooth-coefficients} and Proposition~\ref{prop:correction_est} to derive the second inequality.
	The discrete Grönwall inequality now shows that
	\begin{equation}\label{eq:semi-bound}
		\E{\norm{\widetilde Y_m^f}_{\dot H^\ga}^p}
		\le C\LR{ 1 + \E{\norm{X_0}_{\dot H^\ga}^p} }
		<\infty.
              \end{equation}
	Thus, $\E{\norm{\widetilde Y_m^f}_{\dot H^\ga}^p}<\infty$ for all $p\in(0,8]$,
	and $\E{\norm{\widetilde Y_m^a}_{\dot H^\ga}^p}<\infty$ follows analogously.

	To show the second part, we use again~\eqref{eq:semi-fine3} and repeat the above reasoning to find for $p\in[2,4]$ that
	\be\label{eq:semi_diff}
	\begin{split}
		\E{\norm{Y_m^f - \widetilde Y_m^f}_H^p}
		&\le
		C \E{\norm{Y_0^f - \widetilde Y_0^f}_H^p}
		+
		C \mfd t \sum_{j=0}^{2m-1}
		\E{\norm{F(Y_{j/2}^f)-F(\widetilde Y_{j/2}^f)}_H^p} \\
		&\quad+C \mfd t\sum_{j=0}^{2m-1}
		\E{\norm{G(Y_{j/2}^f) - G(\widetilde Y_{j/2}^f)}_{\LHS(\cU, H)}^p }  \\
		&\quad+C \mfd t^{1+p/2}\sum_{j=0}^{2m-1}
		\E{\norm{\cG(Y_{j/2}^f) - \cG(\widetilde Y_{j/2}^f)}_{\LHS(\LHS(\cU), H)}^p }\\
		&\le
		C \LR{ \E{\norm{(P_{N_f}-I)X_0}_H^p} + \mfd t \sum_{j=0}^{2m-1}
			\E{\norm{Y_{j/2}^f-\widetilde Y_{j/2}^f}_H^p} } \\
		&\quad+C \mfd t^{1+p/2}\sum_{j=0}^{2m-1}
		\E{\norm{\cG(Y_{j/2}^f) - \cG(\widetilde Y_{j/2}^f)}_{\LHS(\LHS(\cU), H)}^p }.
	\end{split}
	\ee
	The second estimate follows since $F$ and $G$ are Fr\'echet differentiable with globally bounded derivatives.
	To bound the last term in~\eqref{eq:semi_diff}, we recall that $\cG'$ satisifies the linear growth bound derived in Proposition~\ref{prop:correction_est} to obtain by Hölder's inequality and~\eqref{eq:semi-bound}
	\begin{align*}
		\E{\norm{\cG(Y_{j/2}^f) - \cG(\widetilde Y_{j/2}^f)}_{\LHS(\LHS(\cU), H)}^p }
		&\le C\E{
			\max\LR{1, \norm{Y_{j/2}^f}_H, \norm{\widetilde Y_{j/2}^f}_H}^p
			\norm{Y_{j/2}^f-\widetilde Y_{j/2}^f}_H^p} \\
		&\le
		C \E{\norm{Y_{j/2}^f-\widetilde Y_{j/2}^f}_H^{2p}}^{\nicefrac{1}{2}}.
	\end{align*}
	Assumption~\ref{ass:approximation}~\ref{item:strong} implies that the semi-discrete approximation satisfies for $p\in[2,4]$ that
	\begin{align*}
		\E{\norm{Y\LR{\frac{j\gd t}{2}}-\widetilde Y_{j/2}^f}_H^{2p}}^{\nicefrac{1}{2}}\le C\LR{\mfd t^{p/2} + K_f^{-p\gb}}.
	\end{align*}
	Together with the first part of Theorem~\ref{thm:stability} and the strong error from Assumption~\ref{ass:approximation}~\ref{item:strong} this shows
	\be\label{eq:semi_diff2}
	\begin{split}
		&\E{\norm{\cG(Y_{j/2}^f) - \cG(\widetilde Y_{j/2}^f)}_{\LHS(\LHS(\cU), H)}^p } \\
		&\quad\le  C\LR{\E{\norm{Y_{j/2}^f-Y\LR{\frac{j\gd t}{2}}}_H^{2p}}^{\nicefrac{1}{2}}
			+ \E{\norm{Y\LR{\frac{j\gd t}{2}}-\widetilde Y_{j/2}^f}_H^{2p}}^{\nicefrac{1}{2}} } \\
		&\quad\le C\LR{\mfd t^{p/2} + N_f^{-p\widetilde \ga} + K_f^{-p\gb}}.
	\end{split}
	\ee
	Substituting~\eqref{eq:semi_diff2} back to~\eqref{eq:semi_diff} now yields with Assumption~\ref{ass:approximation}~\ref{item:subspace_approx}
	\bee
	\begin{split}
		\E{\norm{Y_m^f - \widetilde Y_m^f}_H^p}
		&\le C \LR{  N_f^{-p\widetilde \ga} +
			\mfd t \sum_{j=0}^{2m-1} \E{\norm{Y_{j/2}^f-\widetilde Y_{j/2}^f}_H^p}
			+ \mfd t^{p/2}\LR{\mfd t^{p/2} + N_f^{-p\widetilde \ga}+ K_f^{-p\gb} } } \\
		&\le C \LR{  \mfd t^p + N_f^{-p\widetilde \ga} + K_f^{-p\del{2}\gb} +
			\mfd t \sum_{j=0}^{2m-1} \E{\norm{Y_{j/2}^f-\widetilde Y_{j/2}^f}_H^p} }.
	\end{split}
	\eee
	The claim for $p\in[2,4]$ follows by applying the discrete Grönwall inequality, and the estimate for $p\in(0,2)$ is then obtained immediately by Hölder's inequality.
\end{proof}

Lemma~\ref{lem:semi-error} enables us to derive a difference equation on the fine approximation $Y_\cdot^f$.

\begin{prop}\label{prop:fine-error}
	Let Assumptions~\ref{ass:SPDE} and~\ref{ass:approximation} hold.
	Then, for any $m=0,\dots, M-1$ there holds that
	\begin{equation}
		\begin{split}\label{eq:fine_expansion}
			Y_{m+1}^f
			&=: r(\mfd tA_{N_f})^2P_{N_f}\LR{ Y_m^f
				+ F(Y_m^f)\gD t + G(Y_m^f)\gD_m W_{K_f} + \cG(Y_m^f)\gD_m\cW_{n,K_f   } } \\
			&\quad- r(\mfd tA_{N_f})^2P_{N_f}\cG(Y_m^f)
			\LR{\mfd_{m+1/2}W_{K_f}\otimes\mfd_mW_{K_f} - \mfd_m W_{K_f}\otimes\mfd_{m+1/2}W_{K_f}}
			+ \Xi_m^f + O_m^f,
		\end{split}
	\end{equation}
	where $\Xi_m^f, O_m^f:\gO\to H$ are random variables such that
	\begin{align*}
		&\E{\norm{\Xi_m^f}_H^2}\le C\mfd t^2
		\LR{M^{-\min(\ga, 2)} + N_f^{-2\widetilde \ga} + K_f^{-4\gb}},\\
		&\E{O_m^f\,\big|\cF_{t_m}}=0\quad\text{and}\quad
		\E{\norm{O_m^f}_H^2}\le C\mfd t
		\LR{M^{-\min(\ga, 2)} + N_f^{-2\widetilde \ga} + K_f^{-4\gb}}.
	\end{align*}
	The constant $C$ is independent of $M, N_f$ and $K_f$.
\end{prop}

\begin{proof}
	Assume for simplicity that Assumptions~\ref{ass:SPDE} and~\ref{ass:approximation} hold with $\ga\in[1,2]$.
	Equations~\eqref{eq:fine2} and~\eqref{eq:fine1} show that
	\begin{align*}
		Y_{m+1}^f
		&= r(\mfd tA_{N_f})^2P_{N_f}\LR{ Y_m^f
			+ F(Y_m^f)\gD t + G(Y_m^f)\mfd_m W_{K_f} + \cG(Y_m^f)\mfd_m\cW_{m,K_f} } \\
		&\quad+ r(\mfd tA_{N_f})P_{N_f}\LR{F(Y_{m+1/2}^f)\gD t
			+ G(Y_{m+1/2}^f)\mfd_{m+1/2} W_{K_f} + \cG(Y_{m+1/2}^f)\mfd_{m+1/2}\cW_{m,K_f}},
	\end{align*}
	and by Equation \eqref{eq:tensor_increment} we have
	\begin{equation*}
		\mfd_m\cW_{m,K_f} = \gD_m\cW_{m,K_f} - \mfd_{m+1/2}\cW_{m,K_f}
		- \mfd_{m+1/2}W_{K_f}\otimes \mfd_mW_{K_f}
		- \mfd_mW_{K_f}\otimes \mfd_{m+1/2}W_{K_f}.
	\end{equation*}
	Rearranging some terms yields
	\begin{align*}
		Y_{m+1}^f
		&= r(\mfd tA_{N_f})^2P_{N_f}\LR{ Y_m^f
			+ F(Y_m^f)\gD t + G(Y_m^f)\gD_m W_{K_f} + \cG(Y_m^f)\gD_m\cW_{n,K_f   } }\\
		&\quad- r(\mfd tA_{N_f})^2P_{N_f}\cG(Y_m^f)
		\LR{\mfd_{m+1/2}W_{K_f}\otimes\mfd_mW_{K_f} - \mfd_m W_{K_f}\otimes\mfd_{m+1/2}W_{K_f}}\\
		&\quad+ r(\mfd tA_{N_f})P_{N_f}[F(Y_{m+1/2}^f) - r(\mfd tA_{N_f})P_{N_f} F(Y_m^f)]\mfd t \\
		&\quad+r(\mfd tA_{N_f})P_{N_f}\left[G(Y_{m+1/2}^f) - r(\mfd tA_{N_f})P_{N_f} G(Y_m^f)\right]\mfd_{m+1/2} W_{K_f}\\
		&\quad-2r(\mfd tA_{N_f})P_{N_f}
		r(\mfd tA_{N_f})P_{N_f}\cG(Y_m^f)\mfd_m W_{K_f}\otimes\mfd_{m+1/2}W_{K_f}\\
		&\quad+ r(\mfd tA_{N_f})P_{N_f}
		\left[\cG(Y_{m+1/2}^f)-r(\mfd tA_{N_f})P_{N_f}\cG(Y_m^f)\right]\mfd_{m+1/2}\cW_{m,K_f}.
	\end{align*}
	The first two lines in the above equation correspond to the first two terms on the right hand side in~\eqref{eq:fine_expansion}, and we label the remaining terms via
	\begin{align*}
		\tI_m^f&:=[F(Y_{m+1/2}^f) - r(\mfd tA_{N_f})P_{N_f} F(Y_m^f)]\mfd t, \\
		\tII_m^f&:=\left[G(Y_{m+1/2}^f) - r(\mfd tA_{N_f})P_{N_f} G(Y_m^f)\right]\mfd_{m+1/2} W_{K_f}
		- 2r(\mfd tA_{N_f})P_{N_f}\cG(Y_m^f)\mfd_m W_{K_f}\otimes\mfd_{m+1/2}W_{K_f} \\
		\tI\tII_m^f&:=\left[\cG(Y_{m+1/2}^f)-r(\mfd tA_{N_f})P_{N_f}\cG(Y_m^f)\right]\mfd_{m+1/2}\cW_{m,K_f},
	\end{align*}
	to obtain
	\be\label{eq:error_split}
	\begin{split}
		Y_{m+1}^f
		&= r(\mfd tA_{N_f})^2P_{N_f}\LR{ Y_m^f
			+ F(Y_m^f)\gD t + G(Y_m^f)\gD_m W_{K_f} + \cG(Y_m^f)\gD_m\cW_{n,K_f   } }\\
		&\quad- r(\mfd tA_{N_f})^2P_{N_f}\cG(Y_m^f)
		\LR{\mfd_m W_{K_f}\otimes\mfd_{m+1/2}W_{K_f}
			-\mfd_{m+1/2}W_{K_f}\otimes\mfd_mW_{K_f}}\\
		&\quad+ r(\mfd tA_{N_f})P_{N_f}\left[\tI_m^f + \tII_m^f + \tI\tII_m^f\right].
	\end{split}
	\ee

	We split the first term $\tI_m^f$ further into
	\begin{align*}
		\tI_m^f
		=
		\left[F(Y_{m+1/2}^f) - F(r(\mfd tA_{N_f}) Y_m^f)
		+ F(r(\mfd tA_{N_f}) Y_m^f) - r(\mfd tA_{N_f})P_{N_f} F(Y_m^f)\right]\mfd t
		=:\LR{\tI_{m}^{f,1} + \tI_{m}^{f,2}}\mfd t.
	\end{align*}
	A first order Taylor expansion of $\tI_{m}^{f,1}$ then yields for some $\xi_m^1\in H$
	\be \label{eq:F1-bound}
	\begin{split}
		\tI_{m}^{f,1} &= F(Y_{m+1/2}^f) - F(r(\mfd tA_{N_f}) Y_m^f) \\
		&=
		F'(\xi_m^1)\LR{Y_{m+1/2}^f-r(\mfd tA_{N_f}) Y_m^f}\\
		&=
		F'(\xi_m^1)r(\mfd tA_{N_f})
		\LR{ F(Y_m^f)\gD t + G(Y_m^f)\mfd_m W_{K_f} + \cG(Y_m^f)\mfd_m\cW_{m,K_f} }.
	\end{split}
	\ee
	As $F, G$ and $\cG$ are of linear growth and $\bE(\|Y_m^f\|_H^2)<\infty$ is uniformly bounded by Theorem~\ref{thm:stability}, we have
	\begin{equation*}
		\tI_{m}^{f,1}=\Xi_{m}^{f,1}+O_{m}^{f,1},
	\end{equation*}
	where $\Xi_{m}^{f,1}, O_{m}^{f,1}:\gO\to H$ are random variables such that $\bE(\|\Xi_{m}^{f,1}\|_H^2)\le C\mfd t^2$, $\bE(O_{m}^{f,1}\,|\cF_{t_m})=0$ and $\bE(\|O_{m}^{f,1}\|_H^2)\le C\mfd t$ holds for an independent constant $C>0$.
	To bound $\tI_{m}^{f,2}$, we use first order Taylor expansions around $Y_m^f$ and $\widetilde Y_m^f$ to show that for some $\xi_m^2, \widetilde \xi_m^2\in H$ there holds
	\begin{equation}\label{eq:F2-bound}
		\begin{split}
			\tI_{m}^{f,2}
			&=
			F(r(\mfd tA_{N_f}) Y_m^f) - r(\mfd tA_{N_f})P_{N_f} F(Y_m^f)\\
			&=
			F(Y_m^f)
			+
			F'(\xi_m^2)
			\Big[ r(\mfd tA_{N_f}) Y_m^f - Y_m^f\Big]
			- r(\mfd tA_{N_f})P_{N_f} F(Y_m^f)\\
			&=
			\Big[I- r(\mfd tA_{N_f}) P_{N_f}\Big]F(Y_m^f)
			+
			F'(\xi_m^2)
			\Big[ r(\mfd tA_{N_f})  - I \Big]Y_m^f  \\
			&=
			\Big[I- r(\mfd tA_{N_f}) P_{N_f}\Big]\left(F(\widetilde Y_m^f) +
			F'({\widetilde \xi}_m^2) (Y_m^f - \widetilde Y_m^f)\right)
			+
			F'(\xi_m^2)
			\Big[ r(\mfd tA_{N_f})  - I \Big]
			\left(\widetilde Y_m^f + Y_m^f - \widetilde Y_m^f \right).
		\end{split}
	\end{equation}
	Lemmas~\ref{lem:rational-approx} and~\ref{lem:semi-error}
	thus show together with Items~\ref{item:FG-diff} and~\ref{item:smooth-coefficients} in Assumption~\ref{ass:SPDE} that
	\begin{align*}
		\bE(\|\tI_{m}^{f,2}\|_H^2)
		\le C \left(\mfd t^\ga + \mfd t^2 + N_f^{-2\widetilde \ga} + K_f^{-4\gb}\right)
		\le C \left(\mfd t^\ga + N_f^{-2\widetilde \ga} + K_f^{-4\gb}\right).
	\end{align*}
	As $\| r(\mfd tA_{N_f})P_{N_f}\|_{\cL(H)}\le 1$ by Lemma~\ref{lem:rational-approx} this in turn shows that
	\begin{equation}\label{eq:aux1}
		r(\mfd tA_{N_f})P_{N_f}\tI_m^f = r(\mfd tA_{N_f})P_{N_f}\LR{\tI_{m}^{f,1}+\tI_{m}^{f,2}} \mfd t
		=
		\Xi_{m}^{f, \tI}+O_{m}^{f, \tI},
	\end{equation}
	where $\Xi_{m}^{f, \tI}, O_{m}^{f, \tI}:\gO\to H$ are random variables such that for an independent $C>0$ there holds
	\begin{equation*}
		\E{\norm{\Xi_{m}^{f, \tI}}_H^2}\le C\mfd t^2\LR{\mfd t^\ga + N_f^{-2\widetilde \ga}+ K_f^{-4\gb}},
		\quad
		\E{O_{m}^{f, \tI}\,\big|\cF_{t_m}}=0\quad\text{and}\quad
		\E{\norm{O_{m}^{f, \tI}}_H^2}\le C\mfd t^3.
	\end{equation*}

	We expand the second term $\tII_m^f$ in~\eqref{eq:error_split} via
	\be\label{eq:G_split}
	\begin{split}
		\tII_m^f&=
		\left[G(Y_{m+1/2}^f) - G(r(\mfd tA_{N_f})Y_m^f)\right]\mfd_{m+1/2} W_{K_f}
		- 2r(\mfd tA_{N_f})P_{N_f}\cG(Y_m^f)\mfd_m W_{K_f}\otimes\mfd_{m+1/2}W_{K_f}  \\
		&\quad+\left[G(r(\mfd tA_{N_f})Y_m^f) - r(\mfd tA_{N_f})P_{N_f} G(Y_m^f)\right]\mfd_{m+1/2} W_{K_f} \\
		&=:\tII_{m}^{f,1} + \tII_{m}^{f,2}.
	\end{split}
	\ee
	We observe that $\bE(\tII_{m}^{f,1}|\cF_{t_m})=0$ and obtain by a second order Taylor expansion of $G$ around $r(\mfd tA_{N_f})Y_m^f$ that for some $\xi_m^\tII\in H$ it holds
	\begin{align*}
		\tII_{m}^{f,1}
		&=
		G'(r(\mfd tA_{N_f}) Y_m^f)
		\left(Y_{m+1/2}^f - r(\mfd tA_{N_f})Y_m^f\right)\mfd_{m+1/2} W_{K_f} \\
		&\quad+
		\frac{1}{2}G''(\xi_m^\tII)
		\left(Y_{m+1/2}^f-r(\mfd tA_{N_f}) Y_m^f,\, Y_{m+1/2}^f-r(\mfd tA_{N_f}) Y_m^f\right)
		\mfd_{m+1/2} W_{K_f} \\
		&\quad- 2r(\mfd tA_{N_f})P_{N_f}\cG(Y_m^f)\mfd_m W_{K_f}\otimes\mfd_{m+1/2}W_{K_f}  \\
		&=
		G'(r(\mfd tA_{N_f}) Y_m^f)r(\mfd tA_{N_f})P_{N_f}
		\left( F(Y_m^f)\mfd t+G(Y_m^f)\mfd_m W_{K_f} + \cG(Y_m^f)\mfd_m \cW_{m,K_f}\right)
		\mfd_{m+1/2} W_{K_f}  \\
		&\quad+
		\frac{1}{2}G''(\xi_m^\tII)
		\left(Y_{m+1/2}^f-r(\mfd tA_{N_f}) Y_m^f,\, Y_{m+1/2}^f-r(\mfd tA_{N_f}) Y_m^f\right)
		\mfd_{m+1/2} W_{K_f} \\
		&\quad- r(\mfd tA_{N_f})P_{N_f} G'(Y_m^f)(P_{N_f} G(Y_m^f)\mfd_m W_{K_f})\mfd_{m+1/2}W_{K_f} \\
		&=
		G'(r(\mfd tA_{N_f}) Y_m^f)r(\mfd tA_{N_f})P_{N_f}
		\left( F(Y_m^f)\mfd t + \cG(Y_m^f)\mfd_m \cW_{m,K_f}\right)
		\mfd_{m+1/2} W_{K_f}  \\
		&\quad +\frac{1}{2}G''(\xi_m^\tII)
		\left(Y_{m+1/2}^f-r(\mfd tA_{N_f}) Y_m^f,\, Y_{m+1/2}^f-r(\mfd tA_{N_f}) Y_m^f\right)
		\mfd_{m+1/2} W_{K_f} \\
		&\quad +G'(r(\mfd tA_{N_f}) Y_m^f)
		\left(r(\mfd tA_{N_f})P_{N_f} G(Y_m^f)\mfd_m W_{K_f}\right)\mfd_{m+1/2} W_{K_f} \\
		&\quad -r(\mfd tA_{N_f})P_{N_f} G'(Y_m^f)(P_{N_f} G(Y_m^f)\mfd_m W_{K_f})\mfd_{m+1/2} W_{K_f} \\
		&=: \widetilde{\tI}_{m}^{f,1} + \widetilde{\tII}_{m}^{f,1} + \widetilde{\tIII}_{m}^{f,1}.
	\end{align*}
	The second identity follows by Proposition~\ref{prop:correction_est} since
	\begin{equation*}
		\cG(Y_m^f)\mfd_m W_{K_f}\otimes\mfd_{m+1/2}W_{K_f}
		=
		G'(Y_m^f)(P_{N_f} G(Y_m^f)\mfd_m W_{K_f})\mfd_{m+1/2}W_{K_f}.
	\end{equation*}
	As $F$ and $\cG$ are of linear growth (see Proposition~\ref{prop:correction_est}) and $G'$ is bounded, it follows by the independence of $Y_m$, $\mfd_m \cW_{m,K_f}$ and $\mfd_{m+1/2} W_{K_f}$, Lemma~\ref{lem:ito-BDG} and Theorem~\ref{thm:stability} that
	\begin{equation*}
		\E{\norm{\widetilde{\tI}_{m}^{f,1}}_H^2}
		\le C\LR{1+\E{\norm{Y_m^f}_H^2}}\mfd t^3
		\le C\mfd t^3.
	\end{equation*}
	Similarly, Lemma~\ref{lem:diff} yields with the boundedness of $G''$ that
	\begin{equation*}
		\E{\norm{\widetilde{\tII}_{m}^{f,1}}_H^2}
		\le C\E{\norm{Y_{m+1/2}^f-r(\mfd tA_{N_f}) Y_m^f}_H^4}\mfd t
		\le C\mfd t^3.
	\end{equation*}

	We use Taylor expansion to split the integrand in $\widetilde{\tIII}_{m}^{f,1}$ for some $\widetilde \xi_m^\tIII\in H$ into
	\begin{align*}
		&G'(r(\mfd tA_{N_f}) Y_m^f)
		\left(r(\mfd tA_{N_f})P_{N_f} G(Y_m^f)\mfd_m W_{K_f}\right)
		-r(\mfd tA_{N_f})P_{N_f} G'(Y_m^f)(P_{N_f} G(Y_m^f)\mfd_m W_{K_f}) \\
		&\quad=
		\left[G'(r(\mfd tA_{N_f}) Y_m^f) - G'(Y_m^f)\right]
		\left(r(\mfd tA_{N_f})P_{N_f} G(Y_m^f)\mfd_m W_{K_f}\right)\\
		&\qquad+G'(Y_m^f)
		\left(r(\mfd tA_{N_f})P_{N_f} G(Y_m^f)\mfd_m W_{K_f}\right)
		- G'(Y_m^f)(P_{N_f} G(Y_m^f)\mfd_m W_{K_f}) \\
		&\qquad+G'(Y_m^f)(P_{N_f} G(Y_m^f)\mfd_m W_{K_f})
		-r(\mfd tA_{N_f})P_{N_f} G'(Y_m^f)(P_{N_f} G(Y_m^f)\mfd_m W_{K_f}) \\
		&\quad=
		G''(\widetilde \xi_m^\tIII)\left[(r(\mfd tA_{N_f}) - I) Y_m^f, r(\mfd tA_{N_f})P_{N_f} G(Y_m^f)\mfd_m W_{K_f}\right]\\
		&\qquad+G'(Y_m^f)\left[(r(\mfd tA_{N_f}) - I) P_{N_f} G(Y_m^f)\mfd_m W_{K_f}\right] \\
		&\qquad+(I-r(\mfd tA_{N_f})P_{N_f}) G'(Y_m^f)(P_{N_f} G(Y_m^f)\mfd_m W_{K_f}).
	\end{align*}

	We then use Assumption~\ref{ass:SPDE}~\ref{item:FG-diff} together with Corollary~\ref{cor:ito-BDG} and Theorem~\ref{thm:stability} to estimate
	\begin{align*}
		\E{\norm{\widetilde{\tIII}_{m}^{f,1}}_H^2}
		\le
		&3\bE\left(\left\|
		G''(\widetilde \xi_m^\tIII)\left[(r(\mfd tA_{N_f}) - I) Y_m^f, r(\mfd tA_{N_f})P_{N_f} G(Y_m^f)\mfd_m W_{K_f}\right]
		\right\|_{\LHS(\cU, H)}^2\right)\mfd t\\
		&+ 3\bE\left(\left\|
		G'(Y_m^f)\left[(r(\mfd tA_{N_f}) - I) P_{N_f} G(Y_m^f)\mfd_m W_{K_f}\right]
		\right\|_{\LHS(\cU, H)}^2\right)\mfd t \\
		&+ 3\bE\left(\left\|
		(I-r(\mfd tA_{N_f})P_{N_f}) G'(Y_m^f)(P_{N_f} G(Y_m^f)\mfd_m W_{K_f})
		\right\|_{\LHS(\cU, H)}^2\right)\mfd t \\
		\le
		&C \bE\left(
		\|(r(\mfd tA_{N_f}) - I) Y_m^f\|_H^2
		\left\|
		\left(r(\mfd tA_{N_f})P_{N_f} G(Y_m^f)\mfd_m W_{K_f}\right)
		\right\|_H^2\right)\mfd t \\
		&+
		C\bE\left(\left\|
		(r(\mfd tA_{N_f}) - I) P_{N_f} G(Y_m^f)
		\right\|_H^2\right)\mfd t^2\\
		&+
		C\E{\norm{(r(\mfd tA_{N_f})P_{N_f}-I) G'(Y_m^f)}_{\cL(H, \LHS(\cU, H))}^2
			\norm{G(Y_m^f)\mfd_m W_{K_f}}_H^2 }\mfd t \\
		\le
		& C \E{\norm{(r(\mfd tA_{N_f}) - I)Y_m^f}_H^4}^{\nicefrac{1}{2}}
		\E{\norm{r(\mfd tA_{N_f})P_{N_f} G(Y_m^f)\mfd_m W_{K_f}}_H^4}^{\nicefrac{1}{2}}\mfd t\\
		&+
		C\bE\left(\left\|
		(r(\mfd tA_{N_f})P_{N_f} - I) G(Y_m^f)
		\right\|_{\LHS(\cU, H)}^2\right)\mfd t^2 \\
		&+
		C\bE\left(\left\|
		(r(\mfd tA_{N_f})P_{N_f}-I) G'(Y_m^f)
		\right\|_{\cL(H, \LHS(\cU, H))}^4\right)^{\nicefrac{1}{2}}
		\bE\left(\left\| G(Y_m^f)\mfd_m W_{K_f}\right\|_H^4\right)^{\nicefrac{1}{2}}\mfd t \\
		\le
		& C \E{\norm{(r(\mfd tA_{N_f}) - I)(\widetilde Y_m^f + Y_m^f-\widetilde Y_m^f)}_H^4}^{\nicefrac{1}{2}}
		\mfd t^2\\
		&+
		C\bE\left(\left\|
		(r(\mfd tA_{N_f})P_{N_f} - I) (G(\widetilde Y_m^f)+G(Y_m^f)-G(\widetilde Y_m^f))
		\right\|_{\LHS(\cU, H)}^2\right)\mfd t^2 \\
		&+
		C\bE\left(\left\|
		(r(\mfd tA_{N_f})P_{N_f}-I) (G'(\widetilde Y_m^f)+G'(Y_m^f)-G'(\widetilde Y_m^f))
		\right\|_{\cL(H, \LHS(\cU, H))}^4\right)^{\nicefrac{1}{2}}\mfd t^2.
	\end{align*}
	In the last step we have used that $G$ is of linear growth together with Theorem~\ref{thm:stability}.

	Assumption~\ref{ass:SPDE}~\ref{item:smooth-coefficients} together with Lemmas~\ref{lem:rational-approx} and~\ref{lem:semi-error} then yields
	\bee
	\begin{split}
		&\E{\norm{\widetilde{\tIII}_{m}^{f,1}}_H^2}\\
		&\le
		C \mfd t^2 \LR{
			\E{\norm{(r(\mfd tA_{N_f}) - I)\widetilde Y_m^f}_H^4}^{\nicefrac{1}{2}}
			+
			\E{\norm{Y_m^f - \widetilde Y_m^f}_H^4}^{\nicefrac{1}{2}} } \\
		&\quad+C \mfd t^2 \LR{
			\E{\norm{(r(\mfd tA_{N_f}) - I) G(\widetilde Y_m^f)}_{\LHS(\cU, H)}^2}
			+
			\E{\norm{G(Y_m^f)-G(\widetilde Y_m^f)}_{\LHS(\cU, H)}^2} }
		\\
		&\quad+C \mfd t^2 \LR{
			\E{\norm{(r(\mfd tA_{N_f}) - I) G'(\widetilde Y_m^f)}_{\cL(H, \LHS(\cU, H))}^4}^{\nicefrac{1}{2}}
			+
			\E{\norm{G'(Y_m^f)-G'(\widetilde Y_m^f)}_{\cL(H, \LHS(\cU, H))}^4}^{\nicefrac{1}{2}} } \\
		&\le C \mfd t^2 \left(\mfd t^\ga + N_f^{-2\widetilde \ga} \right)
		\E{\norm{\widetilde Y_m^f}_{\dot H^\ga}^4
			+\norm{G(\widetilde Y_m^f)}_{\LHS(\cU, \dot H^\ga)}^4
			+\norm{G'(\widetilde Y_m^f)}_{\cL(\dot H^\ga, \LHS(\cU, \dot H^\ga))}^4
			+\norm{ Y_m^f-\widetilde Y_m^f }_H^4
		}^{\nicefrac{1}{2}} \\
		&\le C \mfd t^2 \left(\mfd t^\ga + N_f^{-2\widetilde \ga} + K_f^{-4\gb}\right).
	\end{split}
	\eee
	Thus, we obtain for $\tII_{m}^{f,1}:\gO\to H$ that $\bE(\tII_{m}^{f,1}|\cF_{t_m})=0$ and $\bE(\|\tII_{m}^{f,1}\|_H^2)\le C \mfd t^2 \left(\mfd t^\ga + N_f^{-2\widetilde \ga} + K_f^{-4\gb}\right)$.

	For the remaining term $\tII_{m}^{f,2}$ in ~\eqref{eq:G_split}, we
	have that $\bE(\tII_{m}^{f,2}|\cF_{t_m})=0$.
	Morever, with the It\^o isometry and analogous calculations as in~\eqref{eq:F1-bound} and~\eqref{eq:F2-bound} we obtain the bound
	\begin{align*}
		\bE(\|\tII_{m}^{f,2}\|_H)
		&=
		\bE\left(\left\| \left[G(r(\mfd tA_{N_f})Y_m^f) - r(\mfd tA_{N_f})P_{N_f} G(Y_m^f)\right]\mfd_{m+1/2} W_{K_f} \right\|_H^2\right) \\
		&\le C \mfd t \left(\mfd t^\ga + N_f^{-2\widetilde \ga} + K_f^{-4\gb}\right).
	\end{align*}

	To bound the last term $\tI\tII_m^f$ in~\eqref{eq:error_split}, we first observe again that $\bE(\tI\tII_m^f|\cF_n)=0$.
	We further use the BDG inequality in Lemma~\ref{lem:ito-BDG} to obtain for some $\xi_m^{1,\tIII}, \xi_m^{2,\tIII}\in H$ that
	\begin{align*}
		\E{\norm{\tI\tII_m^f}_H^2}=
		&\E{\norm{\left[\cG(Y_{m+1/2}^f)-r(\mfd tA_{N_f})P_{N_f}\cG(Y_m^f)\right]\mfd_{m+1/2}\cW_{m,K_f}}_H^2} \\
		\le
		&C \E{\norm{\cG(Y_{m+1/2}^f)-r(\mfd tA_{N_f})P_{N_f}\cG(Y_m^f)}_{\LHS(\cL_1(\cU), H)}^2} \mfd t^2 \\
		\le
		& C\E{\norm{\cG(Y_{m+1/2}^f)-\cG(r(\mfd tA_{N_f})P_{N_f} Y_m^f)}_{\LHS(\cL_1(\cU), H)}^2} \mfd t^2 \\
		&+C\E{\norm{\cG(r(\mfd tA_{N_f})P_{N_f} Y_m^f)-\cG(Y_m^f)}_{\LHS(\cL_1(\cU), H)}^2} \mfd t^2 \\
		&+C\E{\norm{\cG(Y_m^f)-r(\mfd tA_{N_f})P_{N_f}\cG(Y_m^f)}_{\LHS(\cL_1(\cU), H)}^2} \mfd t^2 \\
		\le
		&C\E{\norm{\cG'(\xi_m^{1,\tIII})(Y_{m+1/2}^f-r(\mfd tA_{N_f})Y_m^f)}_{\LHS(\cL_1(\cU), H)}^2}\mfd t^2 \\
		&+ C\E{\norm{\cG'(\xi_m^{2,\tIII})(I-r(\mfd tA_{N_f}))Y_m^f}_{\LHS(\cL_1(\cU), H)}^2}\mfd t^2 \\
		&+ C\E{\norm{(I-r(\mfd tA_{N_f}))\cG(Y_m^f)}_{\LHS(\cL_1(\cU), H)}^2}\mfd t^2.
	\end{align*}
	Hölder's inequality, Proposition~\ref{prop:correction_est} and now show with similar calculations as for $\widetilde{\tIII}_{m}^{f,1}$ that
	\begin{align*}
		\E{\norm{\tI\tII_m^f}_H^2}
		\le &C\E{1+\norm{\xi_m^{1,\tIII}}_H^4}^{\nicefrac{1}{2}}
		\E{\norm{Y_{m+1/2}^f-r(\mfd tA_{N_f})Y_m^f}_H^4}^{\nicefrac{1}{2}} \mfd t^2 \\
		&+ C\E{1+\norm{\xi_m^{2,\tIII}}_H^4}^{\nicefrac{1}{2}}
		\E{\norm{(I-r(\mfd tA_{N_f}))(\widetilde Y_m^f + Y_m^f-\widetilde Y_m^f)}_H^4}^{\nicefrac{1}{2}} \mfd t^2 \\
		&+ C\E{\norm{(I-r(\mfd tA_{N_f}))
				\LR{\cG(\widetilde Y_m^f) + \cG(Y_m^f)-\cG(\widetilde Y_m^f)}
			}_{\LHS(\cL_1(\cU), H)}^2}\mfd t^2 \\
		\le& C \mfd t^2 \left(\mfd t^\ga + N_f^{-2\widetilde \ga} + K_f^{-4\gb} \right).
	\end{align*}
\end{proof}

A similar result to Proposition~\ref{prop:fine-error} also holds for the antithetic fine approximation $Y_\cdot^a$.

\begin{cor}\label{cor:anti-diff}
	Let Assumptions~\ref{ass:SPDE} and~\ref{ass:approximation} hold.
	Then, for any $m=0,\dots, M-1$ there holds that
	\begin{equation}
		\begin{split}\label{eq:anti_expansion}
			Y_{m+1}^a
			&=: r(\mfd tA_{N_f})^2P_{N_f}\LR{ Y_m^a
				+ F(Y_m^a)\gD t + G(Y_m^a)\gD_m W_{K_f} + \cG(Y_m^a)\gD_m\cW_{m,K_f} } \\
			&\quad+ r(\mfd tA_{N_f})^2P_{N_f}\cG(Y_m^a)
			\LR{\mfd_{m+1/2}W_{K_f}\otimes\mfd_mW_{K_f} - \mfd_m W_{K_f}\otimes\mfd_{m+1/2}W_{K_f}}\\
			&\quad+ \Xi_m^a + O_m^a,
		\end{split}
	\end{equation}
	where $\Xi_m^a, O_m^a:\gO\to H$ are random variables such that
	\begin{align*}
		&\E{\norm{\Xi_m^a}_H^2}\le C\mfd t^2
		\LR{M^{-\min(\ga, 2)} + N_f^{-2\widetilde \ga} + K_f^{-4\gb}},\\
		&\E{O_m^a\,\big|\cF_{t_m}}=0\quad\text{and}\quad
		\E{\norm{O_m^a}_H^2}\le C\mfd t
		\LR{M^{-\min(\ga, 2)} + N_f^{-2\widetilde \ga} + K_f^{-4\gb}}.
	\end{align*}
	The constant $C$ is independent of $\gd t=\nicefrac{T}{2M}, N_f$ and $K_f$.
\end{cor}
\begin{proof}
	Equations~\eqref{eq:anti1} and~\eqref{eq:anti2} show that
	\begin{align*}
		Y_{m+1}^a
		&=: r(\mfd tA_{N_f})^2P_{N_f}\LR{ Y_m^a
			+ F(Y_m^a)\gD t + G(Y_m^a)\gD_m W_{K_f} + \cG(Y_m^a)\gD_m\cW_{n,K_f   } } \\
		&\quad+ r(\mfd tA_{N_f})^2P_{N_f}\cG(Y_m^a)
		\LR{\mfd_{m+1/2}W_{K_f}\otimes\mfd_mW_{K_f} - \mfd_m W_{K_f}\otimes\mfd_{m+1/2}W_{K_f}}\\
		&\quad+r(\mfd tA_{N_f})P_{N_f}\left[I_m^a + \tII_m^a + \tI\tII_m^a\right],
	\end{align*}
	where sign change in the third line is due to the swapping of the increments $\mfd_{m+1/2}W_{K_f}$ and $\gd_{n}W_{K_f}$ in the antithetic estimator. The remainder terms are given by
	\begin{align*}
		I_m^a&:= [F(Y_{m+1/2}^a) - r(\mfd tA_{N_f})P_{N_f} F(Y_m^a)]\mfd t, \\
		\tII_m^a&:=\left[G(Y_{m+1/2}^a) - r(\mfd tA_{N_f})P_{N_f} G(Y_m^a)\right]\mfd_m W_{K_f}
		-2r(\mfd tA_{N_f})P_{N_f}\cG(Y_m^a)\mfd_{m+1/2} W_{K_f}\otimes\mfd_mW_{K_f},\\
		\tI\tII_m^a&:=\left[\cG(Y_{m+1/2}^f)-r(\mfd tA_{N_f})P_{N_f}\cG(Y_m^a)\right]\mfd_m\cW_{m,K_f}.
	\end{align*}
	The claim follows analogously to Proposition~\ref{prop:fine-error}.
\end{proof}

Proposition~\ref{prop:fine-error} and Corollary~\ref{cor:anti-diff} are now combined to derive a similar difference equation for the (antithetic) average $\ol Y_m:=\nicefrac{1}{2}(Y_m^f+Y_m^a)$ for $m=0, \dots, M$.

\begin{prop}\label{prop:avg-diff}
	Let Assumptions~\ref{ass:SPDE} and~\ref{ass:approximation}.
	Then, for any $m=0,\dots, M-1$ there holds that
	\begin{align*}
		\ol Y_{m+1}
		&= r(\mfd tA_{N_f})^2P_{N_f}\LR{\ol  Y_m
			+ F(\ol  Y_m)\gD t + G(\ol  Y_m)\gD_m W_{K_f} + \cG(\ol  Y_m)\gD_m\cW_{m,K_f} }
		+ \ol \Xi_m + \ol O_m,
	\end{align*}
	where $\ol \Xi_m, \ol O_m:\gO\to H$ are random variables such that
	\begin{align*}
		&\E{\norm{\ol \Xi_m}_H^2}\le C\gD t^2
		\LR{M^{-\min(\ga, 2)} + N_f^{-2\widetilde \ga} + K_f^{-4\gb}},\\
		&\E{\ol O_m\,\big|\cF_{t_m}}=0\quad\text{and}\quad
		\E{\norm{\ol O_m}_H^2}\le C\gD  t
		\LR{M^{-\min(\ga, 2)} + N_f^{-2\widetilde \ga} + K_f^{-4\gb}}.
	\end{align*}
	The constant $C$ is independent of $\gD t=\nicefrac{T}{M}, N_f$ and $K_f$.
\end{prop}

\begin{proof}
	Lemma~\ref{prop:fine-error} and Corollary~\ref{cor:anti-diff} show that
	\begin{align*}
		\ol Y_{m+1}
		&= r(\mfd tA_{N_f})^2P_{N_f}\LR{\ol  Y_m
			+ F(\ol  Y_m)\gD t + G(\ol  Y_m)\gD_m W_{K_f} + \cG(\ol  Y_m)\gD_m\cW_{m,K_f} }  \\
		&\quad+ r(\mfd tA_{N_f})^2P_{N_f}\left(\frac{F(Y_m^f)+F(Y_m^a)}{2}-F(\ol Y_m)\right)\gD t\\
		&\quad+ r(\mfd tA_{N_f})^2P_{N_f}\left(\frac{G(Y_m^f)+G(Y_m^a)}{2}-G(\ol Y_m)\right)\gD_m W_{K_f} \\
		&\quad+ r(\mfd tA_{N_f})^2P_{N_f}\left(\frac{\cG(Y_m^f)+\cG(Y_m^a)}{2}-\cG(\ol Y_m)\right)\gD_m\cW_{m,K_f} \\
		&\quad+ \frac{r(\mfd tA_{N_f})^2P_{N_f}}{2}\LR{\cG(Y_m^a)-\cG(Y_m^f)}
		\LR{\mfd_{m+1/2}W_{K_f}\otimes\mfd_mW_{K_f} - \mfd_m W_{K_f}\otimes\mfd_{m+1/2}W_{K_f}}\\
		&\quad+\frac{1}{2}\left(\Xi_m^f + \Xi_m^a+O_m^f + O_m^a\right) \\
		&=: r(\mfd tA_{N_f})^2P_{N_f}\LR{\ol  Y_m
			+ F(\ol  Y_m)\gD t + G(\ol  Y_m)\gD_m W_{K_f} + \cG(\ol  Y_m)\gD_m\cW_{m,K_f} } \\
		&\quad+ \ol{\tI}_m + \ol{\tII}_m +\ol{\tIII}_m +\ol{\tIV}_m
		+\frac{1}{2}\left(\Xi_m^f + \Xi_m^a+O_m^f + O_m^a\right)
	\end{align*}

	To bound the first term $\ol\tI_m$, we use a second order Taylor expansion of $F$ around $\ol Y_m$ together with $\ol Y_m=\frac{1}{2}(Y_m^f+Y_m^a)$ to obtain for some $\xi_m^f, \xi_m^a\in H$
	\begin{align*}
		\frac{F(Y_m^f)+F(Y_m^a)}{2}-F(\ol Y_m)
		&=
		\frac{F''(\xi_m^f)-F''(\xi_m^a)}{4}(Y_m^f-\ol Y_m, Y_m^f-\ol Y_m) \\
		&=
		\frac{F''(\xi_m^f)-F''(\xi_m^a)}{16}(Y_m^f-Y_m^a, Y_m^f-Y_m^a).
	\end{align*}
	Assumption~\ref{ass:approximation}~\ref{item:strong} and the triangle inequality further show that
	\begin{align*}
		\E{\norm{Y_m^f-Y_m^a}_H^4}
		\le
		C\E{\norm{Y_m^f-X(t_m)}_H^4}
		\le
		C\left(\mfd t^2 + N_f^{-4\widetilde\ga} +  K_f^{-4\gb} \right).
	\end{align*}
	We then use the global bound on $F''$ and Jensen's inequality together with Assumption~\ref{ass:approximation}~\ref{item:strong} to derive
	\begin{align*}
		\E{\norm{\ol{\tI}_m}_H^2}
		&\le
		\E{\norm{\frac{F(Y_m^f)+F(Y_m^a)}{2}-F(\ol Y_m)}_H^2}
		\gD t^2 \\
		&\le
		C\E{\norm{Y_m^f-X(m\gD t)+X(m\gD t)-Y_m^a}_H^4}\gD t^2 \\
		&\le
		C\gD t^2 \left(\gD t^2 + N_f^{-4\widetilde\ga} + K_f^{-4\gb} \right).
	\end{align*}

	We observe that $\bE(\ol{\tII}_m\,|\,\cF_{t_m})=0$ holds for the second term, and arrive with It\^o's isometry and similar calculations as for $\ol{\tI}_m$ at
	\begin{align*}
		\E{\norm{\ol{\tII}_m}_H^2}
		\le
		C\gD t \left(\gD t^2 + N_f^{-4\widetilde\ga} + K_f^{-4\gb}\right).
	\end{align*}

	For the next, we first note that $\bE(\ol{\tIII}_m|\,\cF_{t_m})=0$.
	To bound $\ol{\tIII}_m$ in mean-square, we use Proposition~\ref{prop:correction_est} and first order expansion of $\cG$ to obtain for some $\widetilde \xi_m^f, \widetilde \xi_m^a\in H$
	\begin{align*}
		\frac{\cG(Y_m^f)+\cG(Y_m^a)}{2}-\cG(\ol Y_m)
		=
		\frac{\cG'(\widetilde \xi_m^f)-\cG'(\widetilde \xi_m^a)}{2}(Y_m^f-\ol Y_m)
		=
		\frac{\cG'(\widetilde \xi_m^f)-\cG'(\widetilde \xi_m^a)}{4}(Y_m^f-Y_m^a).
	\end{align*}

	Since the intermediate points $\widetilde \xi_m^f, \widetilde \xi_m^a\in H$ are convex combinations of $Y_m^f$ and $Y_m^a$ there holds by Lemma~\ref{lem:ito-BDG}, Proposition~\ref{prop:correction_est} and Theorem~\ref{thm:stability}
	\begin{align*}
		\E{\norm{\ol{\tIII}_m}_H^2}
		&\le C
		\E{\norm{(\cG'(\widetilde \xi_m^f)-\cG'(\widetilde \xi_m^a))(Y_m^f-Y_m^a)}_{\LHS(\cU, H)}^2} \gD t^2 \\
		&\le C\E{\LR{1+\norm{Y_m^f}_H^2+\norm{Y_m^a}_H^2}\norm{Y_m^f-Y_m^a}_H^2} \gD t^2 \\
		&\le C\E{\norm{Y_m^f-Y_m^a}_H^4}^{\nicefrac{1}{2}} \gD t^2\\
		&\le C\gD t^2 \left(\gD t + N_f^{-2\widetilde \ga} + K_f^{-2\gb}\right) \\
		&\le C\gD t\left(\gD t^2 + \gD tN_f^{-2\widetilde \ga} + \gD tK_f^{-2\gb}\right)\\
		&\le C\gD t\left(\gD t^2 + N_f^{-4\widetilde \ga} + K_f^{-4\gb}\right),
	\end{align*}
	where we have used Young's inequality for the final estimate.

	As $\cG$ is of linear growth by Proposition~\ref{prop:correction_est}, we obtain analogously that $\bE(\ol{\tIV}_m|\,\cF_{t_m})=0$ and
	\begin{align*}
		\E{\norm{\ol{\tIV}_m}_H^2}
		\le C\gD t\left(\gD t^2 + N_f^{-4\widetilde \ga} + K_f^{-4\gb}\right).
	\end{align*}
	The dominating remainder terms in the expansion of $\ol Y_{m+1}$ are thus $\Xi_m^f, \Xi_m^a, O_m^f, O_m^a$, and the claim follows from Proposition~\ref{prop:fine-error} and Corollary~\ref{cor:anti-diff}.
\end{proof}

We are now ready to proof our main result.
\begin{proof}[Proof of Theorem~\ref{thm:MainRes}]
	Define $e_{m+1}:=\ol Y_{m+1}-Y_{m+1}^c$ for any $m=0,\dots, M-1$ and assume again without loss of generality that $\ga\in[1,2]$.
	By Proposition~\ref{prop:avg-diff} it holds that
	\begin{equation}\label{eq:ml-diff}
		\begin{split}
			e_{m+1}
			&=  r(\mfd tA_{N_f})^2 \ol Y_m - r(\gD tA_N)Y_m^c\\
			&\quad+\left(r(\mfd tA_{N_f})^2P_{N_f} F(\ol Y_m)
			-r(\gD tA_N)P_NF(Y_m^c)\right)\gD t \\
			&\quad+\left(r(\mfd tA_{N_f})^2P_{N_f} G(\ol Y_m)
			-r(\gD tA_N)P_NG(Y_m^c)\right)\gD_mW_{K_f} \\
			&\quad+\left(r(\mfd tA_{N_f})^2P_{N_f}\cG(\ol Y_m)
			-r(\gD tA_N)P_N\cG(Y_m^c)\right)\gD_m\cW_{m,K_f} \\
			&\quad+r(\gD tA_N)P_NG(Y_m^c)(\gD_m W_{K_f}-\gD_mW_{M} ) \\
			&\quad+ r(\gD tA_N)P_N\cG(Y_m^c)(\gD_m\cW_{m,K_f}-\gD_m\cW_{m,K}) \\
			&\quad+ \ol \Xi_m + \ol O_m.
		\end{split}
	\end{equation}

	We now re-iterate the representation of $\ol Y_m$ and $Y_m^c$ to obtain
	\begin{equation}\label{eq:ml-diff2}
		\begin{split}
			e_{m+1}
			&=r(\mfd tA_{N_f})^{2m} \ol Y_0 - r(\gD tA_N)^mY_0^c \\
			&\quad +\sum_{j=0}^m \left((\mfd tA_{N_f})^{2(m-j+1)}P_{N_f} F(\ol Y_j) - r(\gD tA_N)^{m-j+1}P_NF(Y_j^c)\right)\gD t \\
			&\quad+\sum_{j=0}^m \left(r(\mfd tA_{N_f})^{2(m-j+1)}P_{N_f} G(\ol Y_j)
			-r(\gD tA_N)^{m-j+1}P_NG(Y_j^c)\right)\gD_jW_{K_f} \\
			&\quad+ \sum_{j=0}^m\left(r(\mfd tA_{N_f})^{2(m-j+1)}P_{N_f}\cG(\ol Y_j)
			-r(\gD tA_N)^{m-j+1}P_N\cG(Y_j^c)\right)\gD_j\cW_{m,K_f} \\
			&\quad+\sum_{j=0}^mr(\gD tA_N)^{m-j+1}P_NG(Y_j^c)(\gD_j W_{K_f}-\gD_jW_{M} ) \\
			&\quad+ \sum_{j=0}^mr(\mfd tA_N)^{m-j+1}P_N\cG(Y_j^c)(\gD_j\cW_{m,K_f}-\gD_j\cW_{m,K}) \\
			&\quad+ \sum_{j=0}^mr(\mfd tA_{N_f})^{2(n-j)}(\ol \Xi_j + \ol O_j)\\
			&=:\tI + \sum_{j=0}^m \tII_j+\tI\tII_j+\tIV_j+\tV_j+ \tVI_j+r(\mfd tA_{N_f})^{2(j-1)}(\ol \Xi_j + \ol O_j).
		\end{split}
	\end{equation}

	The first term $\tI$ is bounded Lemma~\ref{lem:rational-approx} and
        Assumption~\ref{ass:approximation}~\ref{item:subspace_approx} since
        $X_0\in L^8(\gO, \dot H^\ga)$ by
        \begin{align*}
		\E{\norm{\tI}_H^2}&\le
		3\E{\norm{(r(\mfd tA_{N_f})^{2m} - r(\gD tA_{N_f})^m)\ol Y_0}_H^2} +
		3\E{\norm{r(\gD tA_{N_f})^m-r(\gD tA_N)^m)\ol Y_0}_H^2} \\
		&\quad +
		3\E{\norm{r(\gD tA_N)^m(\ol Y_0-Y_0^c)}_H^2} \\
		&\le C\E{\norm{(r(\mfd tA_{N_f})^{2m} - r(\gD tA_{N_f})^m)P_{N_f} X_0}_H^2}\\
		&\quad +C \E{\norm{(r(\gD tA_{N_f})^m-r(\gD tA_N)^m)P_{N_f} X_0}_H^2} \\
		&\quad + C\LR{ \E{\norm{P_{N_f} X_0-X_0}_H^2}+\E{\norm{X_0-P_N X_0}_H^2}} \\
		&\le C\LR{M^{-\ga} + N^{-2\widetilde \ga}}.
	\end{align*}

	To bound the terms $\tII_j$, we define the semi-discrete averages $\widetilde Y_j:=\frac{1}{2}\LR{\widetilde Y^f_j + \widetilde Y^a_j}$ for $j=0,\dots, m$. We then obtain for any $j=0,\dots, n$ by a first order Taylor expansion for some $\widetilde\xi_j, \xi_j^c\in H$ that
	\begin{align*}
		\tII_j &=
		\left(r(\mfd tA_{N_f})^{2(m-j+1)}P_{N_f}-r(\gD tA_{N_f})^{m-j+1}P_{N_f}\right)
		\LR{F(\ol Y_j)-F(\widetilde Y_j)+F(\widetilde Y_j)}\gD t \\
		&\quad+\LR{r(\gD tA_{N_f})^{m-j+1}P_{N_f}-r(\gD tA_N)^{m-j+1}P_N}
		\LR{F(\ol Y_j)-F(\widetilde Y_j)+F(\widetilde Y_j)}\gD t \\
		&\quad+r(\gD tA_N)^{m-j+1}P_N\LR{F(\ol Y_m)-F(Y_m^c)}\gD t \\
		&=\left(r(\mfd tA_{N_f})^{2(m-j+1)}P_{N_f}-r(\gD tA_{N_f})^{m-j+1}P_{N_f}\right)
		\LR{F'(\widetilde\xi_j)(\ol Y_j-\widetilde Y_j)+F(\widetilde Y_j)}\gD t \\
		&\quad+\LR{r(\gD tA_{N_f})^{m-j+1}P_{N_f}-r(\gD tA_N)^{m-j+1}P_N}
		\LR{F'(\widetilde\xi_j)(\ol Y_j-\widetilde Y_j)+F(\widetilde Y_j)}\gD t \\
		&\quad+r(\gD tA_N)^{m-j+1}P_N\LR{F'(\xi_j^c)(\ol Y_j - Y_j^c)}\gD t.
	\end{align*}
	Lemmas~\ref{lem:rational-approx} and~\ref{lem:semi-error} then show together with Assumption~\ref{ass:approximation}~\ref{item:smooth-coefficients} that
	\begin{align*}
		\E{\norm{\tII_j}_H^2}
		&\le C \gD t^2\E{\norm{\ol Y_j-\widetilde Y_j}_H^2}\\
		&\quad+C \gD t^2\E{\norm{\left(r(\mfd tA_{N_f})^{2(m-j+1)}P_{N_f}-r(\gD tA_{N_f})^{m-j+1}P_{N_f}\right)F(\widetilde Y_j)}_H^2} \\
		&\quad+
		C \gD t^2\E{\norm{\left(r(\gD tA_{N_f})^{m-j+1}P_{N_f}-S_{N_f}((m-j+1)\gD t)\right)F(\widetilde Y_j)}_H^2} \\
		&\quad+
		C \gD t^2\E{\norm{\left(S_{N_f}((m-j+1)\gD t)-r(\gD tA_N)^{m-j+1}P_N\right)F(\widetilde Y_j)}_H^2
			+\norm{e_j}_H^2 } \\
		&\le C \gD t^2\E{\gD t^2 + N_f^{-2\widetilde \ga} + K_f^{-4\gb} + (\gD t^{\ga}+N_f^{-2\widetilde \ga})\|F(\widetilde Y_j))\|_{\dot H^\ga} +  \|e_j\|_H^2}\\
		&\le C \gD t^2\left(M^{-\ga}+N_f^{-2\widetilde \ga} + K_f^{-4\gb} + \E{\norm{e_j}_H^2}\right).
	\end{align*}

	By Lemma~\ref{lem:ito-BDG} and with similar calculations as for $\tII_m^e$ we further obtain
	\begin{align*}
		\E{\norm{\tI\tII_j}_H^2} + \E{\norm{\tIV_j}_H^2}
		\le C \gD t\left(M^{-\ga}+N_f^{-2\widetilde \ga} + K_f^{-4\gb} + \E{\norm{e_j}_H^2}\right).
	\end{align*}

	The fifth term $\tV_j$ is bounded by Corollary~\ref{cor:ito-BDG} and Theorem~\ref{thm:stability} via
	\begin{equation*}
		\E{\norm{\tV_j}_H^2}
		= \gD t \E{\norm{G(Y_j^c)}_{\LHS(\cU, H)}^2}\sum_{j=K+1}^{K_f}
		\eta_j
		\le C \gD t \E{1+\norm{Y_j^c}_H^2}\sum_{j=K+1}^{K_f}
		\eta_j
		\le C\gD tK^{-2\gb},
	\end{equation*}
	where we have used that $\eta_j=\cO(j^{-(1+\eps)-2\gb})$ for arbitrary
        small $\eps>0$ in the last step, cf.
        Assumption~\ref{ass:approximation} and the subsequent note.
	Similarly, Lemma~\ref{lem:ito-BDG}, Proposition~\ref{prop:correction_est} and Theorem~\ref{thm:stability} show that
	\begin{equation*}
		\E{\norm{\tVI_j}_H^2}
		\le C \gD t^2 \E{\norm{\cG(Y_j^c)}_{\LHS(\cL_1(\cU), H)}^2}\sum_{j=K+1}^{K_f}
		\eta_j^2
		\le C\gD t^2K^{-4\gb}.
	\end{equation*}

	Now we finally observe that $\E{Z\,|\cF_j}=0$ for $Z\in\{\tI\tII_j,\dots,\tVI_j\}$ and every $j=0,\dots,n$, and thus obtain with the estimates on $\tI, \tII_j,\dots, \tVI_j$ that
	\begin{align*}
		\E{\norm{e_{m+1}}_H^2}
		&\le C\E{\norm{\tI}_H^2}
		+ Cm \left(\sum_{j=1}^m \E{\norm{\tII_j}_H^2} + \E{\norm{\ol\Xi_j}_H^2}\right) \\
		&\quad + C\sum_{j=1}^m \LR{
			\E{\norm{\tI\tII_j}_H^2} + \E{\norm{\tIV_j}_H^2} + \E{\norm{\tV_j}_H^2} +\E{\norm{\tVI_j}_H^2} + \E{\norm{\ol O_j}_H^2}} \\
		&\le C \LR{ M^{-\ga} +N^{-2\widetilde \ga} +
			\gD t\LR{\sum_{j=1}^m M^{-\ga}+N_f^{-2\widetilde \ga} + K^{-2\gb} + \E{\norm{e_j}_H^2}}}.
	\end{align*}
	The claim now follows by the discrete Grönwall inequality.
\end{proof}

\section{Proof of Theorem~\ref{thm:mlmc-comp} -- \MMLMC{} Complexity }
\label{sec:app3}

\begin{proof}[Proof of Theorem~\ref{thm:mlmc-comp}]
	Fix $\ell=1,\dots, L$.
	By Theorem~\ref{thm:MainRes} and~\eqref{eq:balance} there holds that
	\begin{equation*}
		\max_{m=0,\dots, M}\E{\norm{\ol Y_m^\ell-Y_m^{c, \ell-1}}_H^2}
		\le C\LR{M_{\ell-1}^{\min(\ga,2)} + N_{\ell-1}^{-2\widetilde \ga} + K_{\ell-1}^{-2\gb}}
		= C M_{\ell-1}^{\min(\ga,2)} .
	\end{equation*}
	Now let $Y^{f,L}$ and $Y^{a,L}$ denote the fine approximation and its antithetic counterpart, respectively, on the finest level $L$.
	The bias of the MLMC estimator is then bounded due to Assumption~\ref{ass:mlmc}~\ref{item:weak} by
	\begin{align*}
		\abs{\E{\Psi(X_T) - \ol\Psi_L}}
		\le C M_L^{-1+\gd}.
	\end{align*}
	Using a first order Taylor expansion of $\Psi\in C_b^2(H,\bR)$ around
        \del{$\ol Y_M^\ell$}\add{\(Y_{M_{\ell-1}}^{c,\ell-1}\)} and
        Theorem~\ref{thm:MainRes} and Corollary~\ref{cor:strong-with-milstein}
        show that for some $\xi^{f,\ell}, \xi^{a,\ell}\in H$ the variance decay on each
        level may be bounded by \rdel{\begin{align*} \var(\ol \Psi_\ell - \Psi_{\ell-1}^c)
          &\le \frac{1}{4}\E{\LR{\Psi(Y_{M_\ell}^{f,\ell}) + \Psi(Y_{M_\ell}^{a,\ell}) - 2\Psi(Y_{M_{\ell-1}}^{c,\ell-1})}^2}\\
          &\le \E{\LR{ \Psi(\ol Y_{M_\ell}^\ell) - \Psi(Y_{M_{\ell-1}}^{c,\ell-1}) +
            \frac{\Psi'(\xi^{f,\ell})}{2}\LR{Y_{M_\ell}^{f,\ell}-\ol Y_{M_\ell}^\ell} +
            \frac{\Psi'(\xi^{a,\ell})}{2}\LR{Y_{M_\ell}^{a,\ell}-\ol Y_{M_\ell}^\ell}
            }^2}\\
          &\le C\E{\norm{\ol Y_{M_\ell}^\ell - Y_{M_{\ell-1}}^{c,\ell-1}}_H^2} \\
          &\le C M_\ell^{-\min(\ga,2)}.
	\end{align*}}
      \add{\begin{align*}
        \var(\ol \Psi_\ell - \Psi_{\ell-1}^c)
        &\le \frac{1}{4}\E{\LR{\Psi(Y_{M_\ell}^{f,\ell}) + \Psi(Y_{M_\ell}^{a,\ell}) - 2\Psi(Y_{M_{\ell-1}}^{c,\ell-1})}^2}\\
        &\le
          \E{
          \LR{\Psi'(Y_{M_\ell}^{c,\ell-1})\LR{\ol Y_{M_\ell}^\ell - Y_{M_{\ell-1}}^{c,\ell-1}} +
          \sum_{b \in \lbrace a,f\rbrace}
          \frac{\Psi''(\xi^{b,\ell})}{2}\LR{Y_{M_\ell}^{b,\ell}- Y_{M_\ell}^{c,\ell-1}}\LR{Y_{M_\ell}^{b,\ell}-Y_{M_\ell}^{c,\ell-1}}}^2}\\
        &\le C\E{\norm{\ol Y_{M_\ell}^\ell - Y_{M_{\ell-1}}^{c,\ell-1}}_H^2}
        + C \sum_{b \in \lbrace a,f\rbrace} \E{\norm{Y_{M_\ell}^{b,\ell} - Y_{M_{\ell-1}}^{c,\ell-1}}_H^4}\\
        &\le C M_\ell^{-\min(\ga,2)} + C M_{\ell}^{-2}.
      \end{align*}}

	Finally, Assumption~\ref{ass:mlmc}~\ref{item:comlexity} yields a cost per sample of $\ol\Psi^\ell$ given by
	$\cC_\ell\le C M_\ell^{1+\gg}$.
	As $M_\ell=M_02^\ell$, \cite[Theorem 2.1]{giles2015multilevel} yields the claim.

\end{proof}
 
\addcontentsline{toc}{section}{References}
\bibliographystyle{abbrv}

\end{document}